\documentclass[12pt]{amsart}
\usepackage{latexsym,amsmath,amssymb}
\usepackage{tikz}\usetikzlibrary{matrix,arrows,calc}

\theoremstyle{plain}
\newtheorem{thm}{Theorem}[section]
\newtheorem{main}{Theorem}\renewcommand{\themain}{\Alph{main}}
\newtheorem{lem}[thm]{Lemma}
\newtheorem{prop}[thm]{Proposition}
\newtheorem{cor}[thm]{Corollary}
\newtheorem{fact}[thm]{Fact}
\theoremstyle{definition}
\newtheorem{defn}[thm]{Definition}
\newtheorem{exmp}[thm]{Example}
\newtheorem{rem}[thm]{Remark}

\newcommand{\size}[1]{|#1|}
\newcommand{\wt}{\widetilde}
\newcommand{\onto}{\twoheadrightarrow}
\newcommand{\into}{\hookrightarrow}
\newcommand{\figsize}{\footnotesize}

\newcommand{\Z}{\mathbb{Z}}
\newcommand{\R}{\mathbb{R}}
\newcommand{\C}{\mathbb{C}}
\newcommand{\quat}{\mathbb{H}}
\newcommand{\sph}{\mathbb{S}}

\newcommand{\fix}{\textsc{Fix}}
\newcommand{\aff}{\textsc{Aff}}
\newcommand{\isom}{\textsc{Isom}}
\newcommand{\spin}{\textsc{Spin}}
\newcommand{\symp}{\textsc{Sp}}
\newcommand{\unt}{\textsc{U}}
\newcommand{\sunt}{\textsc{SU}}
\newcommand{\orth}{\textsc{O}}
\newcommand{\sorth}{\textsc{SO}}
\newcommand{\refl}{\textsc{Refl}}
\newcommand{\cox}{\textsc{Cox}}

\newcommand{\sym}{\textsc{Sym}}
\newcommand{\braid}{\textsc{Braid}}
\newcommand{\cat}{\textsc{CAT}}
\newcommand{\norm}{\textsc{Norm}}
\newcommand{\real}{\textsc{Real}}
\newcommand{\imag}{\textsc{Imag}}

\tikzstyle{RedSolid}=[thick,red,-latex]
\tikzstyle{RedDotted}=[thick,dotted,red!50,-latex,>=stealth, bend right=5]
\tikzstyle{BlackDotted}=[thick,dotted,black!50, bend right=5]
\tikzstyle{BlueDotted}=[thick,dotted,blue!50!white, bend right=5]
\tikzstyle{GreenPoly}=[color=green!30,fill=green!30,join=bevel]
\def\greyA{black!50}  
\def\greenA{black!30!green!30!white}
\definecolor{br1}{rgb}{.90,0.2,0.1}
\definecolor{bb1}{rgb}{0.1,0.2,.70}
\definecolor{bg1}{rgb}{0.1,0.7,0.2}
\definecolor{bo1}{rgb}{.9,.6,.1}
\begin{document}

\title[An elegant complement complex]{A complex euclidean reflection
  group with an elegant complement complex} \date{\today}

\author{Ben Cot\'e}
\address{Department of Mathematics, Bridgewater State University,  Bridgewater, MA 02324} 
\email{cote.bn@gmail.com}

\author{Jon McCammond}
\address{Department of Mathematics, University of California, Santa Barbara, CA 93106} 
\email{jon.mccammond@math.ucsb.edu}

\subjclass[2010]{}


\begin{abstract}
  The complement of a hyperplane arrangement in $\C^n$ deformation
  retracts onto an $n$-dimensional cell complex, but the known
  procedures only apply to complexifications of real arrangements
  (Salvetti) or the cell complex produced depends on an initial choice
  of coordinates (Bj\"orner-Ziegler).  In this article we consider the
  unique complex euclidean reflection group acting cocompactly by
  isometries on $\C^2$ whose linear part is the finite complex
  reflection group known as $G_4$ in the Shephard-Todd classification
  and we construct a choice-free deformation retraction from its
  hyperplane complement onto an elegant $2$-dimensional complex $K$
  where every $2$-cell is a euclidean equilateral triangle and every
  vertex link is a M\"obius-Kantor graph.  Since $K$ is non-positively
  curved, the corresponding braid group is a $\cat(0)$ group, despite
  the fact that there are non-regular points in the hyperplane
  complement, the action of the reflection group on $K$ is not free,
  and the braid group is not torsion-free.
\end{abstract}
\maketitle

\section*{Introduction}

The complement of a hyperplane arrangement in $\C^n$ is obtained by
removing the union of its hyperplanes.  When the arrangement under
consideration is a complexified version of a real arrangement, there
is a classical construction due to Salvetti that provides a
deformation retraction onto an $n$-dimensional cell complex now known
as the \emph{Salvetti complex} of the arrangement
\cite{Sa87}. Bj\"orner and Ziegler extended Salvetti's construction so
that it works for an arbitrary complex hyperplane arrangement, but
their construction depends on an initial choice of a coordinate system
\cite{BjZi92}.  In this article we deformation retract the complement
of a specific infinite affine hyperplane arrangement in $\C^2$ onto an
elegant $2$-dimensional piecewise euclidean complex that involves no
choices along the way.  The arrangement we consider is the set of
hyperplanes for the reflections in a complex euclidean reflection
group that we denote $\refl(\wt G_4)$.  This is the unique complex
euclidean reflection group acting cocompactly by isometries on $\C^2$
whose linear part is the finite complex reflection group known as
$G_4$ in the Shephard-Todd classification.

\begin{main}[Complement complex]\label{main:complement}
  The hyperplane complement of $\refl(\wt G_4)$ deformation retracts
  onto a non-positively curved piecewise euclidean $2$-complex $K$ in
  which every $2$-cell is an equilateral triangle and every vertex
  link is a M\"obius-Kantor graph.
\end{main}

The essence of our construction is easy to describe.  We use the set
of $0$-dimensional hyperplane intersections to form Voronoi cells and
then construct a deformation retraction from the hyperplane complement
onto the portion of the Voronoi cell structure contained in the
complement.  For the group $\refl(\wt G_4)$ all of the Voronoi cells
are isometric and their shape is that of the regular $4$-dimensional
polytope known as the $24$-cell.  The $0$-dimensional intersection at
the center of each Voronoi cell means that as a first step one can
remove its interior by radially retracting onto its $3$-dimensional
polytopal boundary built out of regular octahedra.  This procedure
works for this particular complex euclidean reflection group but it
appears that this is one of the few cases where it can be carried out
without significant modifications.  See Remark~\ref{rem:other}.

Next we use the complement complex $K$ to study the structure of the
braid group of $\refl(\wt G_4)$ acting on $\C^2$.  Recall that for any
group $G$ acting on a space $X$ a point $x\in X$ is said to be
\emph{regular} when its $G$-stabilizer is trivial, the \emph{space of
  regular orbits} is the quotient of the subset of regular points by
the free $G$-action and the \emph{braid group of $G$ acting on $X$} is
the fundamental group of the space of regular orbits.  The name
``braid group'' alludes to the fact that when the symmetric group
$\sym_n$ acts on $\C^n$ by permuting coordinates, the braid group of
this action is Artin's classical braid group $\braid_n$.  For complex
spherical reflection groups, one consequence of Steinberg's theorem is
that the hyperplane complement is exactly the set of regular points
\cite{St64, Le04, LeTa09}.  For complex euclidean reflection groups
the two spaces can be distinct and they are distinct in this case.

\begin{main}[Isolated fixed points]\label{main:fixed-points}
  The space of regular points for the complex euclidean reflection
  group $\refl(\wt G_4)$ acting on $\C^2$ is properly contained
  in its hyperplane complement because of the existence of isolated
  fixed points.
\end{main}

Concretely, for every vertex $v$ in the $2$-complex $K$ located inside
the hyperplane complement there is a non-trivial group element that
fixes $v$ and acts as the antipodal map in the coordinate system with
$v$ as its origin.  Moreover, the set of isolated fixed points that
form the vertices of $K$ are the only non-regular points contained in
the hyperplane complement.  Let $\braid(\wt G_4)$ denote the braid
group of $\refl(\wt G_4)$ acting on $\C^2$.  The well-behaved geometry
of $K$ and the isolated fixed points in the hyperplane complement lead
to an unusual mix of properties for a braid group of a reflection
group.

\begin{main}[Braid group]\label{main:braid-group}
  The group $\braid(\wt G_4)$ is a $\cat(0)$ group and it
  contains elements of order~$2$.
\end{main}

The group $\braid(\wt G_4)$ is a $\cat(0)$ group because it acts
properly discontinuously and cocompactly by isometries on the
$\cat(0)$ universal cover of $K$ and it has elements of order~$2$ that
are caused by the stabilizers of the isolated fixed points in the
hyperplane complement.  Since every finitely generated Coxeter group
is a $\cat(0)$ group that contains $2$-torsion, this combination is
not unusual in the broader world of $\cat(0)$ groups. However, torsion
is unusual in the braid group of a reflection group.  The braid groups
of finite complex reflection groups are torsion-free \cite{Be15}, as
are the braid groups of complexified euclidean Coxeter groups, also
known as euclidean Artin groups or affine Artin groups
\cite{McSu-artin-euclid}.  In fact, it is conjectured that the braid
groups of all complexified Coxeter groups, i.e. all Artin groups, are
torsion-free \cite{GoPa12}.  Thus, this example is a departure from
the norm.

The article is structured from general to specific.  We begin with
basic definitions and results about general complex spherical and
complex euclidean reflection groups.  Then we restrict attention to
complex dimension at most two and describes how quaternions can be
used to give efficient linear-like descriptions of arbitrary
isometries of the complex euclidean plane.  Next, we describe the
$4$-dimensional regular polytope known as the $24$-cell, and
investigate the natural action of $\refl(G_4)$ on this polytope.  The
main tool is a novel visualization technique that makes it easy to
understand the isometries of the $4$-dimensional regular polytopes.
Finally, the last part of the article describes the complex euclidean
reflection group $\refl(\wt G_4)$ in detail and proves our three main
results.

\section{Complex spherical reflection groups\label{sec:spherical}}

This section reviews the definition and classification of the complex
spherical reflection groups.  Recall that in geometric group theory
one seeks to understand groups via their actions on metric spaces and
that the connection between the two is particularly close when the
action is geometric in the following sense.

\begin{defn}[Geometries and geometric actions]\label{def:geometric}
  A metric space $X$ is called a \emph{proper metric space} or a
  \emph{geometry} when for every point $x \in X$ and for every
  positive real $r$, the closed metric ball of radius $r$ around $x$
  is a compact subspace of $X$ and a group $G$ acting on a geometry
  $X$ is said to act \emph{geometrically} when the action of $G$ on
  $X$ is properly discontinuous and cocompact by isometries.
\end{defn}

The first geometry we wish to consider is that of the unit sphere in a
complex vector space with a positive definite inner product.

\begin{defn}[Complex spherical geometry]\label{def:spherical-geometry}
  Let $V = \C^n$ be an $n$-dimensional complex vector space.  When $V$
  comes equipped with a positive definite hermitian inner product that
  is linear in the second coordinate and conjugate linear in the
  first, we say that $V$ is a \emph{complex spherical geometry}.  For
  an appropriate choice of basis, the inner product of vectors $v$ and
  $w$ in $V$ can be written as $\langle v, w \rangle = v^* w =
  \sum_{i=1}^n \bar{v}_i w_i$ where $v$ and $w$ are viewed as column
  vectors or as $n$ by $1$ matrices and for any matrix $A$, $A^*$
  denotes its \emph{adjoint} or conjugate transpose.  The
  \emph{length} of a vector $v$ is $|v|=\sqrt{\langle v,v\rangle}$ and
  \emph{unit vectors} are those of length~$1$.  The linear
  transformations of $V$ that preserve the inner product are the
  \emph{unitary transformations}, they form the \emph{unitary group
    $U(V)$ or $U(n)$} and they are precisely those linear
  transformations that preserve the sphere of unit vectors in $V$ and
  its complex structure, sending complex lines in $V$ to complex lines
  and the corresponding oriented circles in the unit sphere
  $\sph^{2n-1}$ to oriented circles.
\end{defn}

A complex reflection is an elementary isometry of such a geometry.

\begin{defn}[Complex reflections]\label{def:complex-reflections}
  Let $V$ be a complex spherical geometry.  Vectors in $V$ are
  \emph{orthogonal} or \emph{perpendicular} when their inner product
  is $0$ and the \emph{orthogonal complement} of a vector $v$ is the
  set of all vectors perpendicular to $v$.  A \emph{complex
    reflection} $r$ is a unitary transformation of $V$ that multiples
  some unit vector $v$ by a unit complex number $z \in \C$ and
  pointwise fixes the vectors in the orthogonal complement of $v$.
  The formula for the reflection $r = r_{v,z}$ is $r(w) = w - (1-z)
  \langle v,w\rangle v$.  The reflection $r$ has finite order if and
  only if $z = e^{ai}$ where $a$ is a rational multiple of $\pi$ and
  when this occurs we say that $r$ is a \emph{proper} complex
  reflection.  The name refers to the fact that the action of the
  cyclic subgroup generated by $r$ on the unit sphere is properly
  discontinuous if and only if $r$ is a proper reflection.  Since
  properly discontinuous actions require proper complex reflections,
  only proper reflections are considered and we drop the adjective.
  When the complex number $z$ is of the form $z=e^{\frac{2\pi}{m}i}$
  for some positive integer $m$, the complex reflection $r_{v,z}$ is
  said to be \emph{primitive}, and note that every finite cyclic
  subgroup generated by a single proper complex reflection contains a
  unique primitive generator.
\end{defn}

We are interested in groups generated by complex reflections.

\begin{defn}[Complex spherical reflection groups]\label{def:sph-refl-grps}
  A group $G$ is called a \emph{complex spherical reflection group} if
  it is generated by complex reflections acting on a complex spherical
  geometry $V$ so that the action restricted to the unit sphere in $V$
  is geometric in the sense of Definition~\ref{def:geometric}.  Such
  groups are also known as \emph{finite complex reflection groups}.
  If there is an orthogonal decomposition $V = V_1 \oplus V_2$
  preserved by all of the elements of $G$, then $G$ is
  \emph{reducible} and it is \emph{irreducible} when such a
  decomposition does not exist.  In 1954 Shephard and Todd completely
  classified the irreducible complex spherical reflection groups.
  There is a single triply-indexed infinite family $G(de,e,r)$ where
  $d$, $e$ and $r$ are positive integers that they split into $3$
  subcases $G_1$, $G_2$ and $G_3$ based on some additional properties
  and $34$ exceptional cases that they label $G_4$ through $G_{37}$
  \cite{ShTo54,Co76}. Since this article discusses both reflection
  groups and the corresponding braid groups, we use the symbol $G_k$
  with $k$ between $4$ and $37$ to indicate a \emph{Shephard-Todd
    type} analogous to the Cartan-Killing types that index so many
  objects in Lie theory and we write $\refl(G_k)$ to denote the
  exceptional complex spherical reflection group of type $G_k$
  identified by Shephard and Todd.
\end{defn}

The main group of interest here is a euclidean extension of the
smallest exceptional complex spherical reflection group $\refl(G_4)$.

\section{Complex euclidean reflection groups\label{sec:euclidean}}

The transition from complex spherical to complex euclidean geometry
involves replacing the underlying vector space and its distinquished
origin with the corresponding affine space where all points are on an
equal footing.

\begin{defn}[Affine space]
  For any vector space $V$, the abstract definition of the
  corresponding \emph{affine space} is a set $E$ together with a
  simply transitive $V$ action on $E$.  The elements of $E$ are
  \emph{points}, the elements of $V$ are \emph{vectors} and we write
  $x + v$ for the image of point $x \in E$ under the action of $v\in
  V$.  For each linear subspace $U \subset V$ and point $x\in E$ there
  is an \emph{affine subspace} $x + U = \{x + v \mid v \in U\} \subset
  E$ that collects the images of $x$ under the action of the vectors
  in $U$ and the functions $f \colon E \to E$ that send affine
  subspaces to affine subspaces are \emph{affine maps}.  For each
  vector $v \in V$ there is a \emph{translation map $t_v \colon E \to
    E$} that sends each point $x$ to $x+v$ and this is an affine map.
  The collection of all translation maps is an abelian group
  isomorphic to the vector space $V$ under addition and it is a normal
  subgroup of the group $\aff(E)$ of all affine transformations.  If
  we pick a point $x \in E$ as our \emph{basepoint} then every point
  $y$ in $E$ can be labeled by the unique vector $v \in V$ that sends
  $x$ to $y$ so that $E$ based at $x$ is naturally identified with $V$
  and the group of all affine maps can be identified with the
  semidirect product of the translation group and the invertible
  linear transformations of $E$ based at $x$ now identified with $V$.
  In other words, for each point $x \in E$ there is a natural
  isomorphism between the group $\aff(E)$ and the semidirect product
  $V \rtimes GL(V)$.
\end{defn}  

When the vector space $V$ is a complex spherical geometry, it makes
sense to restrict attention to those affine transformations that
preserve the hermitian inner product.

\begin{defn}[Complex euclidean space]
  Let $E$ be an affine space for a complex vector space $V$.  When $V$
  is a complex spherical geometry, then $E$ is a \emph{complex
    euclidean geometry}.  Since an ordered pair $(x,x')$ of points in
  $E$ determines a vector $v_{x,x'} \in V$ that sends $x$ to $x'$, an
  ordered quadruple $(x,x',y,y')$ of points in $E$ determines an
  ordered pair of vectors $(v_{x,x'},v_{y,y'})$ in $V$ to which the
  hermitian inner product can be applied.  An affine map $f\colon E
  \to E$ is called a \emph{complex euclidean isometry} when $f$
  preserves the hermitian inner product of the ordered pair of vectors
  derived from an ordered quadruple of points in $E$.  In other words
  $\langle v_{x,x'},v_{y,y'} \rangle = \langle
  v_{f(x),f(x')},v_{f(y),f(y')} \rangle$ for all $x, x', y, y' \in E$.
  The group of all complex euclidean isometries is denoted $\isom(E)$.
  All translations are complex euclidean isometries and an affine map
  fixing a point $x$ is a complex euclidean isometry if and only if
  the corresponding linear transformation of $V$ is a unitary
  transformation.  Therefore, for each point $x \in E$ there is a
  natural isomorphism between the group $\isom(E)$ and the semidirect
  product $V \rtimes U(V)$ or $\C^n \rtimes U(n)$ once an orthonormal
  coordinate system has been introduced.
\end{defn}

The spherical notion of a complex reflection is extended to complex
euclidean space as follows.

\begin{defn}[Complex euclidean reflection groups]
  An isometry of a complex euclidean space $E$ is called a
  \emph{complex reflection} if it becomes a complex reflection in the
  sense of Definition~\ref{def:complex-reflections} for an appropriate
  choice of origin and identification of the space $E$ with the vector
  space $V$.  For us, a \emph{complex euclidean reflection group} is
  any group generated by complex reflections that acts geometrically
  on a complex euclidean space.  In the literature, the complex
  euclidean reflection groups that act geometrically on $E$ are called
  \emph{crystallographic}.  The image of a complex euclidean
  reflection group $G$ under the projection map from $\isom(E) \to
  U(V)$ is called its \emph{linear part} and the kernel is its
  \emph{translation part}.  In many but not all examples the group $G$
  has the structure of a semidirect product of its linear and
  translation parts.  The group $G$ is called \emph{reducible} or
  \emph{irreducible} depending on the corresponding property of its
  linear part.  Two complex euclidean reflection groups $G$ and $G'$
  acting on complex euclidean spaces $E$ and $E'$ are called
  \emph{equivalent} when there is an invertible affine map from $E$ to
  $E'$ (that need not preserve the complex euclidean metric) so that
  the action of $G$ on $E$ corresponds to the action of $G'$ on $E'$
  under this identification.
\end{defn}

\begin{rem}[Known examples]
  The collection of known inequivalent irreducible complex euclidean
  reflection groups includes $30$ infinite families and $22$ isolated
  examples.  Some of the infinite families have a discrete parameter
  that indicates the dimension of the space on which it acts, some of
  the infinite families have a continuous complex parameter which,
  when varied, produces inequivalent reflection groups that all act on
  the same space, and some have both a discrete and a continuous
  parameter.  The $17$ infinite families with a continuous complex
  parameter correspond exactly to those whose linear part is an
  irreducible finite real reflection group.  There is one such family
  for each simply-laced Cartan-Killing type ($A_n$, $D_n$, $E_6$,
  $E_7$, $E_8$) and multiple families for each of the others ($G_2$
  has $4$, $F_4$ has $3$ and $B_n=C_n$ has $5$ -- except in dimension
  $n=2$ where the identification $\cox(\wt B_2) \cong \cox(\wt C_2)$
  reduces the number of parameterized families from $5$ to $3$).  The
  $7$ families of type $A$, $B=C$ and $D$ have both a continuous
  parameter and a discrete parameter, the $10$ families of type $E$,
  $F$ and $G$ have a continuous parameter only.  Next there are $13$
  infinite families with primitive linear part indexed by a discrete
  parameter but with only one instance in each dimension.  And
  finally, there are $7$ isolated examples with primitive linear part
  that only occur in low dimensions ($3$ in dimension~$1$ and $4$ in
  dimension~$2$) and $15$ isolated examples whose linear part is one
  of the $34$ exceptional complex spherical reflection groups ($5$ in
  dimension~$2$, $7$ in dimension~$3$ and one each in dimensions $4$,
  $5$ and $6$).
\end{rem}

\begin{rem}[Classification]
  The inequivalent irreducible complex euclidean reflection groups
  were essentially classified by Popov in \cite{Po82}.  He established
  many structural results about these groups and gave algorithms in
  each of the various subcases that together could be used to produce
  a complete list.  Some of the details of the computations that
  connect the algorithms with the explicit tables of examples,
  however, were not included and in 2006 Goryunov and Man found an
  isolated example in dimension~$2$ that was not among those listed by
  Popov, thus calling the completeness of the tables into question
  \cite{GoMa06}.
\end{rem}

We write $\refl(\wt G_4)$ to denote the unique complex euclidean
reflection group whose linear part is $\refl(G_4)$. Popov denotes it
$[K_4]$.

\section{Isometries of the complex euclidean line}

The inequivalent complex euclidean reflection groups that act
geometrically on the complex euclidean line are easily classified.  In
this section we review their classification and preview the Voronoi
cell argument in this easy-to-visualize context.

\begin{defn}[Isometries and reflections]
  Every isometry of the complex euclidean line is a function of the
  form $f(x) = e^{ai} x + z$ where $a$ is real and $z$ is an arbitary
  complex number.  And since complex euclidean reflections acting on
  $\C$ must fix some point $z_0$ (i.e. some affine copy of $\C^0$),
  they are precisely those isometries of the form $e^{ai}(x-z_0) +
  z_0$.
\end{defn}

The fact that we are only interested in discrete actions places a
strong restriction on the orders of the complex euclidean reflections
that can be used.

\begin{lem}[Crystallographic]
  If $r$ and $s$ are primitive complex reflections of order $m$ with
  distinct fixed points acting on $\C$, then the action of the group
  they generate is indiscrete unless $m \in \{2,3,4,6\}$.
\end{lem}

\begin{proof}
  The product $t=rs^{-1}$ is a translation and the composition of the
  translations $rtr^{-1}$ and $r^{-1}tr$ is another translation in the
  same direction as $t$ but its translation distance is $2\cos
  \frac{\pi}{m}$ times that of $t$.  In particular, the group of
  translations in this direction act indiscretely unless $\cos
  \frac{\pi}{m}$ is rational, and this is true exactly for $m \in
  \{2,3,4,6\}$.
\end{proof}

The crystallographic restriction makes it easy to classify the complex
euclidean reflection groups that act on the complex euclidean line.

\begin{thm}[Classification]
  If $G$ is a complex euclidean reflection group that acts
  geometrically on the complex euclidean line, then every reflection in
  $G$ has order $2$, $3$, $4$ or $6$ and its reflections of maximal
  order generate~$G$.  When the maximal order is $2$ there is a
  $1$-parameter family of such groups, but when it is $3$, $4$ or $6$
  there is a unique such group up to affine equivalence.
\end{thm}

We include a brief description of each case.

\begin{exmp}[Order $2$]
  When all reflections have order $2$, their fixed points form a
  lattice in $\C$, i.e. a discrete $\Z^2$ subgroup in $\C$ once one of
  these fixed points has been chosen as the origin.  After rescaling
  so that there are fixed points at $0$ and $1$ and no pair of fixed
  points less than $1$ unit apart, the various inequivalent cases are
  described by a third generator fixing a point $z$ with $|z| \geq 1$
  and the real part of $z$ in the interval $[-\frac12,\frac12]$ with
  some identifications along the boundary.
\end{exmp}

\begin{exmp}[Orders $3$, $4$ and $6$]
  For $m=3$, $m=4$ and $m=6$ we start with an equilateral triangle, an
  isosceles right triangle and a triangle with angles $\frac{\pi}{2}$,
  $\frac{\pi}{3}$ and $\frac{\pi}{6}$, respectively.  There is a
  triangular tiling of $\C$ generated by the real reflections in the
  sides of this triangle.  The real reflection groups generated are
  the euclidean Coxeter groups $\cox(\wt A_2)$, $\cox(\wt B_2)$ and
  $\cox(\wt G_2)$, respectively.  In each case, the index $2$ subgroup
  of orientation-preserving isometries is generated by those complex
  reflections rotating through an angle of $\frac{2\pi}{m}$ fixing a
  point where $2m$ triangles met.  These groups are denoted $[K_3(m)]$
  in Popov's notation and $\refl(\wt G_3(m))$ in ours.
\end{exmp}

\begin{rem}[Fixed points and translations]\label{rem:fixed-points}
  Let $G$ be a complex euclidean reflection group acting on $\C$, let
  $T$ be the subgroup of translations, let $T_0$ be the images of the
  origin under the translations in $T$ and let $FP_m$ be the fixed
  points of the primitive reflections of order $m$ (assuming they
  exist) and assume that the origin is fixed by a primitive reflection
  $r$ of order $m$.  The computation $t_v \circ r \circ t_v^{-1} (x) =
  z(x-v)+v = zx + (1-z)v = t_{(1-z)v} \circ r(x)$ with $z =
  e^{\frac{2\pi}{m}i}$ shows that $(1-z) \cdot FP_m = T_0$.  In the
  group $\refl(\wt G_3(6))$, for example, $2 \cdot FP_2 = (1-\omega)
  \cdot FP_3 = FP_6 = T_0$, where $\omega = e^{2\pi i/3}$ is a
  primitive cube-root of unity.
\end{rem}

\begin{defn}[Voronoi cells]
  Let $S$ be a discrete set of points in some euclidean space $E$.
  The \emph{Voronoi cell} around $s$ is the set of points in $E$ that
  are as close to $s$ as they are to any point in $S$.  These regions
  are delineated by the hyperplanes that are equidistant between two
  points in $S$. Thus, the Voronoi cells are euclidean polytopes so
  long as these regions are bounded (as they are in our context).  The
  union of these euclidean polytopes gives the entire euclidean space
  $E$ a piecewise euclidean cell structure that we call the
  \emph{Voronoi cell structure}.  The Voronoi cell structure of a
  complex euclidean reflection group $G$ is the cell structure
  obtained when $S$ is the set of $0$-dimensional intersections of the
  fixed hyperplanes of the complex reflections in $G$.
\end{defn}

As should be clear from its definition, the Voronoi cell structure is
preserved by the complex euclidean group used to create it.

\begin{exmp}[Voronoi cells]
  Let $G$ be one of the complex euclidean reflection groups $\refl(\wt
  G_3(m))$ with $m \in \{3,4,6\}$ and let $S$ be the set of fixed
  points for the reflections in $G$.  The vertices of the Voronoi
  cells in this case are the centers of the inscribed circles of the
  triangles in the corresponding triangular tiling, the edges are
  built out of the altitudes from these centers to the sides of the
  triangles and the Voronoi cells themselves are regular polygons,
  hexagons for $m=3$, squares and octagons for $m=4$ and squares,
  hexagons and dodecagons for $m=6$.  The case $m=3$ is illustrated in
  Figure~\ref{fig:voronoi}.
\end{exmp}

\begin{figure}
  \begin{tikzpicture}[scale=.8]
    \clip (-5,-4) rectangle (5,4);
    \def\sqrtthree{1.732050}
    \def\hexorange#1{
      \draw[orange!70,thick] #1 +(30:\sqrtthree)
      \foreach \a in {30,90,...,360} {node {} -- +(\a:\sqrtthree) } -- cycle; 
      \filldraw[orange!70, opacity=.2] #1 +(30:\sqrtthree) 
      \foreach \a in {30,90,...,360} { -- +(\a:\sqrtthree) } -- cycle;
    }
    \hexorange{(0,0)}
    \foreach \a in {0,60,...,300}{\hexorange{(\a:3)}}
    \foreach \a in {30,90,...,330}{\hexorange{(\a:3*\sqrtthree)}}
    \foreach \a in {60,120,...,360}{\hexorange{(\a:6)}}
    \def\hextri#1#2{
      \draw[help lines] #1+(0:#2) \foreach \a in {60,120,...,300} { -- +(\a:#2) } -- cycle;
      \foreach \a in {0,60,...,300}{\draw[help lines] #1  -- ++(\a:#2);}
    }
    \hextri{(0,0)}{3}
    \foreach \b in {30,90,...,330} {\hextri{(\b:3*\sqrtthree)}{3}}
    \fill[green!70!black] (0,0) circle(1mm);
    \foreach \a in {0,120,240} {
      \fill[red!70!black] (\a:3) circle(1mm);
      \fill[blue!70!black] (\a+60:3) circle(1mm);
      \fill[blue!70!black] (\a:6) circle(1mm);
      \fill[red!70!black] (\a+60:6) circle(1mm);
      \fill[green!70!black] (\a+30:3*\sqrtthree) circle(1mm);
      \fill[green!70!black] (\a-30:3*\sqrtthree) circle(1mm);
    }
    \draw[thick] (-5,-4) rectangle (5,4);
  \end{tikzpicture}
\caption{The Voronoi cell structure for the complex euclidean
  reflection group $\refl(\wt G_3(3))$ is a hexagonal tiling of $\C$
  and the hyperplane complement deformation retracts to its
  $1$-skeleton.\label{fig:voronoi}}
\end{figure}
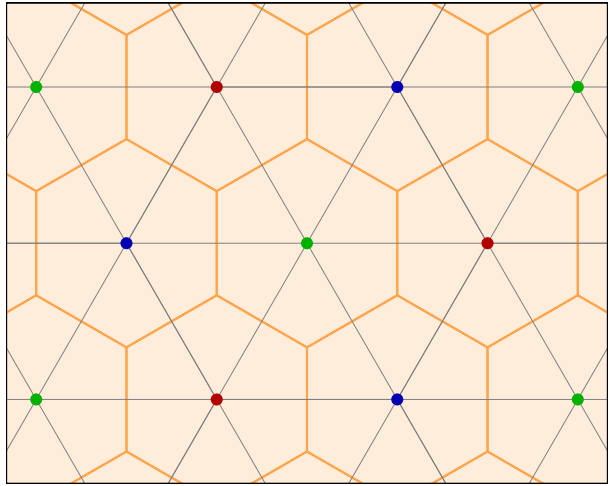

The Voronoi cells can be used to understand the braid groups.

\begin{thm}[Braid groups]
  For $m=3$, $4$ and $6$, the braid group $\braid(\wt G_3(m))$ is
  isomorphic to the free group of rank~$2$.
\end{thm}

\begin{proof}
  In all three cases, once the fixed points of the reflections are
  removed, the remainder deformation retracts to the $1$-skeleton of
  the Voronoi cell structure.  The group acts freely on the
  $1$-skeleton but it does not act transitively on the vertices.  The
  quotient graph has $2$ vertices with $3$ edges connecting them, a
  graph whose fundamental group is the free group of rank~$2$.
\end{proof}

\section{Quaternions and their complex structures\label{sec:quaternion}}

In this section we recall basic properties of the quaternions and
their subalgebras isomorphic to the complex numbers.  The goal is to
establish notation for the quaternions with a specified complex
structure.

\begin{defn}[Quaternions]\label{def:quat}
  Let $\quat$ denote the \emph{quaterions}, the skew field and normed
  division algebra of dimension $4$ over the reals with standard basis
  $\{1,i,j,k\}$ where $i^2 = j^2 = k^2 = ijk = -1$ and $i$, $j$ and
  $k$ pairwise anticommute.  The reals $\R$ are identified with the
  $\R$-span of $1$ inside $\quat$ and they form its center: every real
  is central and every central element is real.  If $q = a+bi+cj+dk$
  with $a,b,c,d \in \R$ then $\real(q)=a$ is its \emph{real part} and
  $\imag(q) = bi+cj+dk$ is its \emph{imaginary part}.  A quaternion is
  \emph{purely imaginary} if its real part is $0$ and \emph{real} if
  its imaginary part is $0$.  The \emph{conjugate} of $q$ is $\bar{q}
  = a - (bi+cj+dk)$, its \emph{norm} $\norm(q) = q\bar{q} = \bar{q}q =
  a^2+b^2+c^2+d^2$ and its \emph{length} $\size{q}$ is the square root
  of its norm. The \emph{distance} between $q$ and $q'$ is the length
  of $q-q'$.  This distance function makes $\quat$ into a
  $4$-dimensional euclidean space with $\{1,i,j,k\}$ as an orthonormal
  basis.  The quaternions in the unit $3$-sphere in $\R^4$ have norm
  $1$ and are the set of \emph{unit quaternions}.  Every nonzero
  quaternion can be \emph{normalized} by dividing by its length.
\end{defn}

The unit quaternions show that the $3$-sphere has a Lie group
structure.  It can be identified with the compact symplectic Lie group
$\symp(1)$, the spin group $\spin(3)$ (the double cover of $SO(3)$) or
special unitary group $\sunt(2)$ once a complex structure has been
chosen.  The quaternions have a canonical copy of the reals and thus a
canonical euclidean structure, but they contain a continuum of
subalgebras isomorphic to $\C$ and a corresponding continuum of ways
to specify a complex spherical structure.

\begin{defn}[Complex subalgebras]\label{def:c-sub}
  For each purely imaginary unit quaternion $u$, $u^2=-1$ and the
  $\R$-span of $1$ and $u$ is a subalgebra of $\quat$ isomorphic to
  the complex numbers with $u$ playing the role of $\sqrt{-1}$.  More
  generally, note that every nonreal quaternion $q_0$ determines a
  complex subalgebra of $\quat$ in which $q_0$ has positive imaginary
  part.  Concretely, the $\R$-span of $1$ and $q_0$ is a complex
  subalgebra and the isomorphism with $\C$ identifies $\sqrt{-1}$ with
  the normalized imaginary part of~$q_0$. We call this the
  \emph{complex subalgebra determined by~$q_0$}.
\end{defn}

The choice of a complex subalgebra determines a complex structure.

\begin{defn}[Complex structures]\label{def:complex-structures}
  Let $q_0$ be a nonreal quaternion and identify $\C$ with the complex
  subalgebra of $\quat$ determined by $q_0$.  The right cosets $q \C$
  of $\C$ inside $\quat$ partition the nonzero quaternions into right
  \emph{complex lines}.  Vector addition and this type of right scalar
  multiplication turn $\quat$ into a $2$-dimensional right vector
  space over this subalgebra $\C$.  In addition, there is a unique
  positive definite hermitian inner product on this $2$-dimensional
  complex vector space so that the unit quaternions have length $1$
  with respect to this inner product.  We call this the \emph{right
    complex structure on $\quat$ determined by $q_0$} and we write
  $\quat_{q_0}$ to denote the quaternions with this choice of complex
  structure.  Note that when $q_1 = a+bq_0$ with $a$ real and $b$
  positive real, $\quat_{q_0}$ and $\quat_{q_1}$ define the same
  complex structure.
\end{defn}

The complex structure used in our computations is $\quat_\omega$ where
$\omega = \frac{-1+i+j+k}{2}$ is a cube root of unity.  The pure unit
quaternion that plays the role of $\sqrt{-1}$ in the chosen complex
subalgebra is $\frac{i+j+k}{\sqrt{3}}$.

\begin{defn}[Unit complex numbers]\label{def:unit-circle}
  Because the complex subalgebra we use in our computations does not
  contain the quaternion $i$, we do not use $i$ as a notation for
  $\sqrt{-1}$ in the distinguished copy of $\C$, but we make an
  exception for the unit complex numbers.  Specifically, we write
  $z=e^{ai}$ with $a$ real for the numbers on the unit circle in $\C$
  even though the chosen copy of $\C$ does not contain the quaternion
  $i$.  Since this misuse of the letter $i$ only occurs as an exponent
  and only in this particular formulation, the improvement in clarity,
  in our opinion, outweighs any potential confusion.
\end{defn}

Those who prefer computations over $\C$ can select an ordered basis
and work with coordinates.  Note that we use the letter $z$ rather
than $q$ when we wish to emphasize that a particular quaternion lives
in the distinguished copy of $\C$.

\begin{defn}[Bases and Coordinates]
  Let $\quat_{q_0}$ be the quaternions with a complex structure.
  Every ordered pair of nonzero quaternions $q_1$ and $q_2$ that
  belong to distinct complex lines form an \emph{ordered basis} of
  $\quat_{q_0}$ viewed as a $2$-dimensional right complex vector
  space.  In particular, their right $\C$-linear combinations $q_1 \C
  + q_2 \C$ span all of $\quat_{q_0}$ and for every $q \in
  \quat_{q_0}$ there are unique \emph{coordinates} $z_1, z_2 \in \C$
  such that $q = q_1 z_1 + q_2 z_2$.  When the basis $\mathcal{B} =
  \{q_1,q_2\}$ is ordered we view the coordinates of $q$ as a column
  vector.  When the complex structure is determined by $j$ and the
  ordered basis $\mathcal{B} = \{1,i\}$, for example, the quaternion
  $q = a + bi + cj + dk$ has coordinates $z_1 = a+cj$ and $z_2 = b+dj$
  because $q = 1(a + cj) + i(b+dj)$.  In other words, inside $\quat_j$
\[q = a + bi + cj + dk = \left[ \begin{array}{c} z_1 \\ z_2 \end{array}
  \right]_\mathcal{B} = \left[ \begin{array}{c} a+cj
    \\ b+dj \end{array} \right]_\mathcal{B}.\]

\end{defn}

\section{Isometries of the complex euclidean plane\label{sec:isometries}}

This section concisely describes each isometry of the complex
euclidean plane using an elementary quaternionic map.  We begin with the
left and right multiplication maps.

\begin{defn}[Spherical maps]\label{def:sph-maps}
  For each quaternion $q$ there is a left multiplication map $L_q(x) =
  qx$ and a right multiplication map $R_q(x) = xq$ from $\quat$ to
  itself and these maps are isometries of the canonical euclidean
  structure of $\quat$ if and only if $q$ has length $1$.  When $q$ is
  not a unit, they are euclidean similarities but not isometries since
  they change lengths.  When $q$ is a unit quaternion, both $L_q$ and
  $R_q$ are orientation preserving euclidean isometries that fix the
  origin, send the unit $3$-sphere to itself and move every point in
  $\sph^3$ the same distance.  For each pair of unit quaternions $q$
  and $q'$, there is a function defined by the composition $f = L_q
  \circ R_{q'} = R_{q'} \circ L_q$ or explicitly by the equation $f(x)
  = q x q'$ that we call a \emph{spherical map}.  Every spherical map
  induces an orientation preserving isometry of $\sph^3$ and every
  orientation preserving isometry of $\sph^3$ can be represented as a
  spherical map in precisely two ways.  The second representation is
  obtained from the first by negating both $q$ and $q'$.  This
  correspondence essentially identifies the topological space $\sph^3
  \times \sph^3$ of pairs of unit quaternions with the Lie group
  $\spin(4)$, the double cover of $\sorth(4)$. For details see
  \cite{CoSm03}.
\end{defn}

The spherical maps that preserve a complex structure are special.

\begin{defn}[Complex spherical maps]
  Once a complex structure is added to the quaternions, only some
  spherical maps preserve this structure and we call those that do
  \emph{complex spherical maps}.  For every unit quaternion $q$ the
  left multiplication map $L_q$ sends the complex lines in
  $\quat_{q_0}$ to complex lines and it is a complex spherical
  isometry.  Right multiplication is different because of the
  noncommutativity of quaternionic multiplication.  The only right
  multiplication maps that sent complex lines to complex lines are
  those of the form $R_z$ where $z$ is number in the chosen complex
  subalgebra and the only isometries among them are those where $z$ is
  a unit.  When $z = e^{ai}$ is unit complex number (i.e. a unit
  quaternion in the complex subalgebra generated by $1$ and $q_0$),
  the map $R_z$ is a complex spherical isometry that stabilizes each
  individual complex line $q\C$ setwise and rotates it by through an
  angle of $a$ radians.  As $z$ varies through the unit complex
  numbers, this motion is called the \emph{Hopf flow}.
\end{defn}

The next proposition records the fact that left multiplication and the
Hopf flow are sufficient to generate all complex spherical isometries.

\begin{prop}[Complex spherical isometries]\label{prop:complex-sph-iso}
  The spherical maps that preserve the complex structure of
  $\quat_{q_0}$ are precisely those of the form $x \mapsto qxz$ where
  $q$ is a unit quaternion and $z$ is a unit complex number in the
  chosen complex subalgebra.
\end{prop}

As with general spherical maps, each complex spherical map can be
represented in two ways because of the equality $qxz = (-q)x(-z)$.
This gives a map from $\sph^3 \times \sph^1 \onto \unt(2)$ with kernel
$\{\pm 1\}$, which corresponds to the short exact sequence $\orth(1)
\into \symp(1) \times \unt(1) \onto \unt(2)$.  For later use we
concretely describe the action of $L_z$ and $R_z$ for any unit complex
number $z$ in some detail.

\begin{rem}[Left and Right]\label{rem:left-right}
  Let $\quat_{q_0}$ be the quaternions with a complex structure, let
  $q_1$ be any unit quaternion orthogonal to both $1$ and $q_0$, and
  let $z = e^{ai}$ with $a$ real be a unit complex number.  Both maps
  $L_z$ and $R_z$ stabilize the complex lines $1\C$ and $q_1\C$
  setwise, but their actions on these lines are slightly different.
  The map $R_z$ rotates both lines through an angle of $a$ radians
  while the map $L_z$ rotates $1\C$ through an angle of $a$ radians
  and the line $q_1\C$ through an angle of $-a$ radians.  The minus
  occurs because $q_1$, being orthogonal to $1$ and $q_0$, is a pure
  imaginary quaternion that commutes with $1$ and anticommutes with
  the pure imaginary part of $q_0$. Thus $z q_1 = q_1 \bar{z}$ and
  $\bar{z} = e^{-ai}$.

  In the ordered basis $\mathcal{B} = \{1,q_1\}$ the hermitian inner
  product is the standard one, $q_1\C$ is the unique complex line that
  is orthogonal to $1\C$, and the maps $R_z$ and $L_z$ can be
  represented as left multiplication by $2\times 2$ matrices over the
  complex numbers on the column vector of coordinates with respect to
  $\mathcal{B}$.  Let $x$ be the quaternion with coordinates $x_1$ and
  $x_2$ with respect to $\mathcal{B}$ so that $x = 1x_1 + q_1 x_2$ with
  $x_1, x_2 \in \C$.  The element $xz = x_1z + q_1 x_2z = zx_1 + q_1 z
  x_2$ because elements in $\C$ commute.  Thus:
  \[ R_z(x) = xz = \left[ \begin{array}{cc} e^{ai} & 0 \\ 0 & e^{ai} \end{array}
    \right] \left[ \begin{array}{c} x_1 \\ x_2 \end{array}
    \right]_\mathcal{B} \]
  On the other hand, the element $zx = zx_1 + z q_1 x_1 = z x_1 + q_1
  \bar{z} x_2$ as discussed above.  Thus:
  \[ L_z(x) = zx = \left[ \begin{array}{cc} e^{ai} & 0 \\ 0 & e^{-ai} \end{array}
    \right] \left[ \begin{array}{c} x_1 \\ x_2 \end{array}
    \right]_\mathcal{B} \]
\end{rem}

Note that the matrix for $R_z$ lies in the center of $\unt(2)$ and the
matrix for $L_z$ lies in $\sunt(2)$.  

\begin{defn}[Complex reflections]\label{def:complex-reflections-q}
  Let $\quat_{q_0}$ be the quaternions with a complex structure, let
  $q_1$ be any unit quaternion orthogonal to both $1$ and $q_0$, and
  let $z = e^{ai}$ with $a$ real be a unit complex number.  The
  complex spherical map $L_z \circ R_z (x) = zxz$ is a complex
  reflection because it fixes the complex line $q_1\C$ pointwise and
  rotates the complex line $\C = 1\C$ through an angle of $2a$
  radians.  In the notation of
  Definition~\ref{def:complex-reflections} this map is $r_{1,z^2}$.
  To create an arbitrary complex spherical reflection $r_{q,z^2}$ with
  $q$ a unit quaternion, it suffices to conjugate $r_{1,z^2}$ by $L_q$
  since the composition $L_q \circ r_{1,z^2} \circ L_{q^{-1}}$ defined
  by the equation $x \mapsto (qzq^{-1})xz$ rotates the complex line
  $q\C$ through an angle of $2a$ and fixes the unique complex line
  orthogonal to $q\C$ pointwise.
\end{defn}

This explicit description makes complex reflections easy to detect.

\begin{prop}[Complex reflections]\label{prop:complex-refl-q}
  Let $\quat_{q_0}$ be the quaternions with a complex structure.  A
  complex spherical map $f(x) = qxz$ with $q$ a unit quaternion and
  $z$ a unit complex number is a complex reflection if and only if
  $\real(q) = \real(z)$.
\end{prop}

\begin{proof}
  Both directions are easy quaternionic exercises.  In one direction
  conjugation by a quaternion does not change its real part so the
  formula given in Definition~\ref{def:complex-reflections-q} shows
  that every complex reflection satisfies $\real(q)=\real(z)$.  In the
  other direction, whenever there are two unit quaternions $q$ and $z$
  with the same real part, there is a third unit quaternion $p$ that
  conjugates $z$ to $q$, and once $qxz$ is rewritten as $(pzp^{-1})xz$
  it is clear that $f$ is the complex reflection $r_{p,z^2}$.
\end{proof}
  
Once translations are included in the discussion, these results about
isometries fixing the origin readily extend to arbitrary isometries of
the complex euclidean plane.

\begin{defn}[Translations]
  For every quaternion $q$ the \emph{translation} map $t_q(x) = x+q$
  is an orientation preserving isometry of the canonical euclidean
  structure of $\quat$.  When a spherical map is combined with
  translation by an arbitrary quaternion $q''$ we call the resulting
  function $f(x) = qxq'+q''$ a \emph{euclidean map}.  As was the case
  with spherical maps, every euclidean map is an orientation
  preserving euclidean isometry and every orientation preserving
  euclidean isometry can be represented as a euclidean map in
  precisely two ways (with the second representation obtained by
  negating $q$ and $q'$).
\end{defn}

Once we allow translations, the quaternions $\quat_{q_0}$ with a
complex structure can be identified as the complex euclidean plane.
The images of the complex lines $q\C$ under translation are called
\emph{affine complex lines} and they are sets of the form $q\C +v$.
Every translation preserves this complex euclidean structure and
Propositions~\ref{prop:complex-sph-iso} and~\ref{prop:complex-refl-q}
extend.

\begin{prop}[Complex euclidean isometries]\label{prop:complex-euc-iso}
  The euclidean maps that preserve the complex euclidean structure of
  $\quat_{q_0}$ are precisely those of the form $x \mapsto qxz +v$
  where $q$ is a unit quaternion and $z$ is a unit complex number in
  the chosen complex subalgebra and $v$ is arbitrary.
\end{prop}

\begin{prop}[Complex euclidean reflections]\label{prop:complex-euc-refl}
  Let $\quat_{q_0}$ denote the quaternions with a complex structure.
  A complex euclidean map $f(x) = qxz +v$ with $q$ a unit quaternion,
  $z$ a unit complex number and $v$ arbitrary is a complex euclidean
  reflection if and only if $f$ has a fixed point and $\real(q) =
  \real(z)$.
\end{prop}

\section{The $24$-cell\label{sec:24-cell}}

This section describes the regular polytope called the $24$-cell and
introduces a novel technique for visualizing its structure.

\begin{defn}[The $24$-cell]\label{def:24-cell}
  The convex hull of the unit quaternions
  \[ \Phi = \{\pm 1, \pm i, \pm j, \pm k \} \cup \left\{\frac{\pm 1 \pm i
    \pm j \pm k}{2}\right\}\] is a $4$-dimensional regular polytope
  known as the \emph{$24$-cell} because it has $24$ regular octahedral
  facets.  The centers of these $24$ euclidean octahedra are at the
  points $\frac{i-j}{2} \cdot \Phi$, where $q \cdot \Phi$ denotes a
  scaled and rotated version of $\Phi$ obtained by left multiplying
  every element of $\Phi$ by a quaternion $q$.  In particular,
  $\frac{i-j}{2} \cdot \Phi$ consists of the $24$ quaternions of the
  form $\frac{\pm u \pm v}{2}$ for $u,v \in \{1,i,j,k\}$ with $u \neq
  v$.  We use $\Phi$ for this set because it is the conventional
  letter used for root systems and the type $D_4$ root system is the
  set $\Phi_{D_4} = (i-j) \cdot \Phi$.  We also note that the
  quaternions in $\Phi$ form a subgroup of $\quat$.
\end{defn}

We use elements in $\Phi$ to define a complex structure on $\quat$.

\begin{defn}[$\omega$ and $\zeta$]
  Let $\omega=\frac{-1+i+j+k}{2}$ and let $\zeta = \frac{1+i+j+k}{2}$,
  and note that $\omega$ is a cube-root of unity, $\zeta$ is a
  sixth-root of unity and $\zeta^2 = \omega$.  For the remainder of
  the article we give the quaternions the complex structure
  $\quat_\omega = \quat_\zeta$.  Since $\Phi$ is a group of order $24$
  and $\zeta$ is an element in $\Phi$ of order $6$, we can partition
  $\Phi$ into the four cosets $q \langle \zeta \rangle$ with $q \in
  \{1,i,j,k\}$.  Thus every element in $\Phi$ is of the form
  $q\zeta^\ell$ with $q \in \{1,i,j,k\}$ and $\ell$ an integer mod $6$
  and $\Phi$ is contained in the union of the four complex lines
  $1\C$, $i\C$, $j\C$ and~$k\C$.
\end{defn}

\begin{figure}
  \begin{center}
    \begin{tikzpicture}[scale=.64]
      \def\R{7cm} 
      \def\ra{3.3cm} 
      \def\rb{2.2cm} 
      \def\rc{1.732cm} 
      \def\ds{.5mm} 
      \foreach \l in {0,1,2} {
        \begin{scope}[shift=(120*\l-30:\R)]
          \filldraw[\greenA] (0,0) circle (3cm);
          \foreach \n in {0,1,2} {\draw[rotate=120*\n,blue!50,dashed] (3,0) arc (0:180:3cm and 1cm);}
          \foreach \n in {0,1,2} {\draw[rotate=120*\n,\greenA,line width=3pt] (210:\rc) arc (240:300:3cm and 1cm);}
          \foreach \n in {0,1,2} {\draw[rotate=120*\n,blue,semithick] (-3,0) arc (180:360:3cm and 1cm);}
          \foreach \n in {0,1,2} {
            \draw (120*\n+33:\rc) edge[RedDotted] (120*\n+87:\rc);
            \draw (120*\n-27:\rc) edge[BlueDotted] (120*\n+27:\rc);
            \filldraw[\greyA] (120*\n+30:\rc) circle (\ds);
            \filldraw[black] (120*\n-30:\rc) circle (\ds);
          }
        \end{scope}
        \begin{scope}[shift=(120*\l+30:\R)]
          \filldraw[\greenA] (0,0) circle (3cm);
          \foreach \n in {0,1,2} {\draw[rotate=120*\n,blue!50,dashed] (-3,0) arc (180:360:3cm and 1cm);}
          \foreach \n in {0,1,2} {\draw[rotate=120*\n,\greenA,line width=3pt] (30:\rc) arc (60:120:3cm and 1cm);}
          \foreach \n in {0,1,2} {\draw[rotate=120*\n,blue,semithick] (3,0) arc (0:180:3cm and 1cm);}
          \foreach \n in {0,1,2} {
            \draw (120*\n+33:\rc) edge[BlueDotted] (120*\n+87:\rc);
            \draw (120*\n-27:\rc) edge[RedDotted] (120*\n+27:\rc);
            \filldraw[black] (120*\n+30:\rc) circle (\ds);
            \filldraw[\greyA] (120*\n-30:\rc) circle (\ds);
          }
        \end{scope}
      }
      
      \foreach \l in {0,1,...,5} {
        \begin{scope}[shift=(60*\l-30:\R)]
          \foreach \x in {0,1,...,5} { 
	    \draw[RedSolid] (60*\x+3:3cm) arc (60*\x+3:60*\x+57:3cm);
	    \filldraw (60*\x:3) circle (\ds);
          }
          \draw (0:\ra) node {\figsize $1$};
          \draw (60:\ra) ++(.1,0) node {\figsize $\zeta$};
          \draw (120:\ra) node {\figsize $\zeta^2$};
          \draw (180:\ra) ++(-.15,0) node {\figsize $-1$};
          \draw (240:\ra) ++(-.2,-.1) node {\figsize $\zeta^4$};
          \draw (300:\ra) ++(.2,-.1) node {\figsize $\zeta^5$};
          \filldraw (0,0) circle (\ds) [color=\greyA];     
        \end{scope}
      }
      
      \begin{scope}[shift=(90:\R)]      
        \draw (90:\rb) node {\figsize $i$};
        \draw (210:\rb) ++(-.2,-.1) node {\figsize $j\zeta^2$};
        \draw (330:\rb) ++(.2,-.1) node {\figsize $k\zeta^4$};
        \draw (270:\rb) ++(-.1,-.1) node[\greyA] {\figsize $-k$};
        \draw (0,.1) node [anchor=north, \greyA] {$\frac{i-k}{2}$};
      \end{scope}	
      
      \begin{scope}[shift=(330:\R)]    
        \draw (90:\rb) ++(.05,.1) node {\figsize $j$};
        \draw (210:\rb) ++(-.2,-.1) node {\figsize $k\zeta^2$};
        \draw (330:\rb) ++(.17,0) node {\figsize $i\zeta^4$};
        \draw (270:\rb) ++(-.1,-.1) node[\greyA] {\figsize $-i$};
        \draw (0,.1) node [anchor=north, \greyA] {$\frac{j-i}{2}$} ;
      \end{scope}
      
      \begin{scope}[shift=(210:\R)]    
        \draw (90:\rb) node {\figsize $k$};
        \draw (210:\rb) ++(-.2,-.1) node {\figsize $i\zeta^2$};
        \draw (330:\rb) ++(.2,0) node {\figsize $j\zeta^4$};
        \draw (270:\rb) ++(-.1,-.1) node[\greyA] {\figsize $-j$};
        \draw (0,.1) node [anchor=north, \greyA] {$\frac{k-j}{2}$} ;
      \end{scope}
      
      \begin{scope}[shift=(270:\R)]    
        \draw (90:\rb) node[\greyA] {\figsize $k$};
        \draw (30:\rb) ++(.3,.1) node[black] {\figsize $j\zeta^5$};
        \draw (150:\rb) ++(-.1,.05) node[black] {\figsize $k\zeta$};
        \draw (270:\rb) ++(-.1,-.1) node[black] {\figsize $-i$};
        \draw (0,.1) node [anchor=north, \greyA] {$\frac{k-i}{2}$} ;
      \end{scope}
      
      \begin{scope}[shift=(150:\R)]      
        \draw (90:\rb) node[\greyA] {\figsize $i$};
        \draw (30:\rb) ++(.33,0) node[black] {\figsize $k\zeta^5$};
        \draw (150:\rb) ++(-.1,.1) node[black] {\figsize $i\zeta$};
        \draw (270:\rb) ++(-.1,-.1) node[black] {\figsize $-j$};
        \draw (0,.1) node [anchor=north, \greyA] {$\frac{i-j}{2}$} ;
      \end{scope}
      
      \begin{scope}[shift=(30:\R)]    
        \draw (90:\rb) ++(.05,.1) node[\greyA] {\figsize $j$};
        \draw (30:\rb) ++(.33,0) node[black] {\figsize $i\zeta^5$};
        \draw (150:\rb) ++(-.1,.1) node[black] {\figsize $j\zeta$};
        \draw (270:\rb) ++(-.1,-.1) node[black] {\figsize $-k$};
        \draw (0,.1) node [anchor=north, \greyA] {$\frac{j-k}{2}$} ;
      \end{scope}
    \end{tikzpicture}
    \end{center}
  \caption{Six lenses that together display the structure of the
    $24$-cell.  Each figure represents a one-sixth lens in the
    $3$-sphere with dihedral angle $\frac{\pi}{3}$ between its front
    and back hemispheres.  They are arranged so that every front
    hemisphere is identified with the back hemisphere of the next one
    when ordered in a counter-clockwise way.\label{fig:lenses}}
  \end{figure}
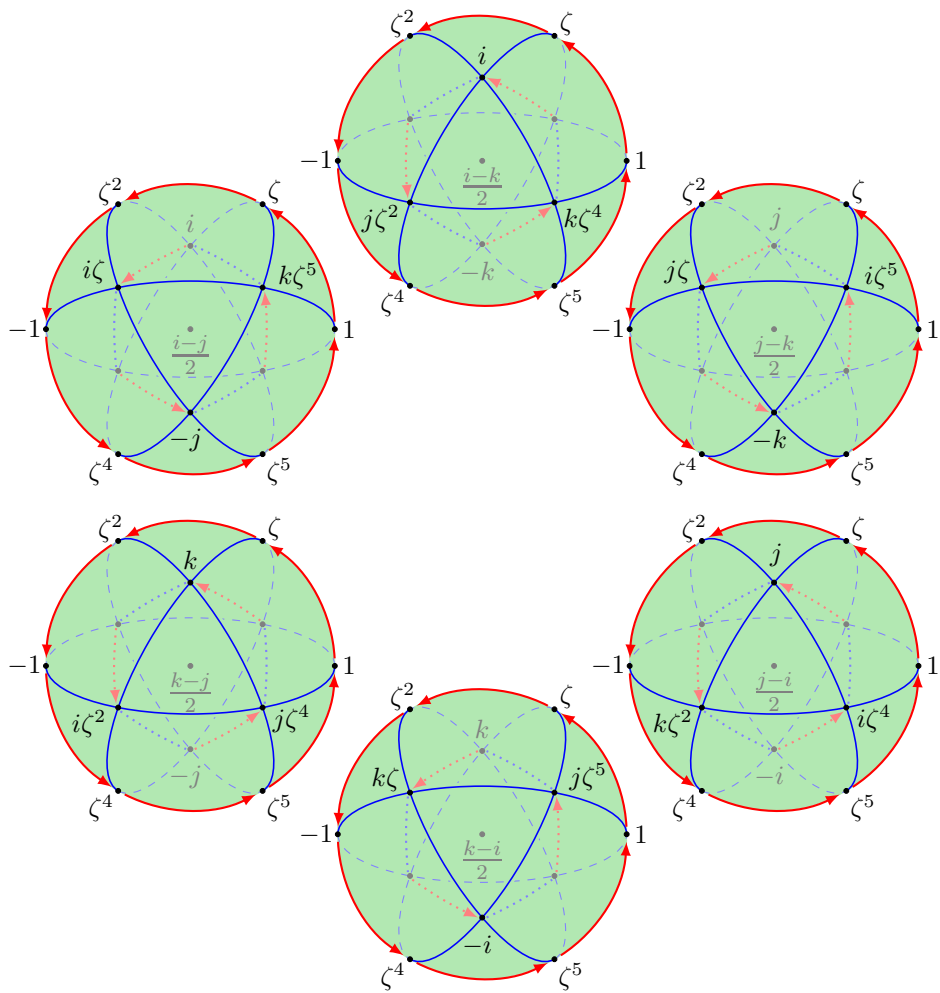

In 2007 John Meier and the second author developed a technique for
visualizing the regular $4$-dimensional polytopes as a union of
spherical lenses that has been very useful for understanding the
various groups that act on these polytopes.  To our knowledge this is
the first time that this technique has appeared in print.

\begin{defn}[Lunes and Lenses]\label{def:lune-lens}
  A \emph{lune} is a portion of a $2$-sphere bounded by two
  semicircular arcs with a common $0$-sphere boundary and its shape is
  completely determined by the angle at this these semicircles meet.
  A lens is a $3$-dimensional analog of a lune.  Concretely, a
  \emph{lens} is a portion of the $3$-sphere determined by two
  hemispheres sharing a common great circle boundary and the shape of
  a lens is completely determined by the dihedral angle between these
  hemispheres along the great circle where they meet.
\end{defn}

In the same way that lunes can be used to display the map of a
$2$-sphere such as the earth in $\R^2$ with very little distortion,
lenses can be used to display a map of the $3$-sphere in $\R^3$ with
very little distortion.

\begin{defn}[$6$ lenses]\label{def:6-lenses}
  To visualize the structure of the $24$-cell it is useful to use the
  $6$ lenses displayed in Figure~\ref{fig:lenses}.  Each of the six
  figures represents one-sixth of the $3$-sphere.  The outside circle
  is a great circle in $\sph^3$, the solid lines live in the
  hemisphere that bounds the front of the lens, the dashed lines live
  in the hemisphere that bounds the back of the lens and the dotted
  lines live in the interior of the lens.  The dihedral angle between
  the front and back hemispheres, along the outside boundary circle is
  $\frac{\pi}{3}$ and all the edges are length $\frac{\pi}{3}$.  The
  six lenses are arranged so that the front hemisphere of each lens is
  identified with the back hemisphere the next one in
  counter-clockwise order.  Each lens contains one complete octahedral
  face at its center and six half octahedra, three bottoms halves
  corresponding to the squares in the front hemisphere and three top
  halves corresponding to the squares in the back hemisphere.  The
  label at the center of each lens is the coordinate of the center of
  the euclidean octahedron spanned by the six nearby vertices. The
  arrows in Figure~\ref{fig:lenses} indicate how the $24$ vertices
  move under the map $R_\zeta$ which right multiplies by $\zeta$.  The
  arrows glue together form four oriented hexagons with vertices $q
  \langle \zeta\rangle$ that live in the four complex lines $q\C$
  where $q$ is $1$, $i$, $j$ or $k$.
\end{defn}

\section{The group $\refl(G_4)$\label{sec:G4}}

In this section the complex spherical reflection group $\refl(G_4)$ is
defined and its natural action on the $24$-cell is investigated.

\begin{defn}[The group $\refl(G_4)$]
  The group $\refl(G_4)$ is defined to be the complex spherical
  reflection group generated by the order~$3$ reflections
  $r_{1,\omega}(x) = \zeta x \zeta$ and $r_{i,\omega}(x) = \zeta^i x
  \zeta$, where $\zeta^i = (-i) \zeta i = i \zeta (-i) = \zeta^{-i}$
  is the conjugation of $\zeta$ by $\pm i$.  For simplicity we
  abbreviate these as $r_1 = r_{1,\omega}$ and $r_i = r_{i,\omega}$.
  The resulting group also includes the order $3$ reflections $r_j =
  r_{j,\omega}(x) = \zeta^j x \zeta$ and $r_k = r_{k,\omega}(x) =
  \zeta^k x \zeta$ as well as the reflections $r_q^2 = r_{q,\omega^2}$
  for $q \in \{1,i,j,k\}$.  It turns out that this group includes the
  map which (left or right) multiplies by $-1$, so it also includes the
  negatives of these eight reflections, which are no longer
  reflections.  Finally, $\refl(G_4)$ contains elements which left
  multiply by $\pm q$ with $q \in \{1,i,j,k\}$.  Thus the full list of
  all $24$ elements in $\refl(G_4)$ is $\{\pm L_q\} \cup \{ \pm r_q \}
  \cup \{ \pm r_q^2 \}$ with $q \in \{1,i,j,k\}$.
\end{defn}

\begin{rem}[Binary tetrahedral group]\label{rem:binary-tetra}
 The group formed by the elements in $\Phi$ is called the \emph{binary
   tetrahedral group} and it can be identified with the group of left
 multiplications $L_q$ with $q \in \Phi$ acting freely on $\Phi$,
 preserving the complex structure $\quat_\omega$.  Its name derives
 from the fact that it is the inverse image of the rotation group of
 the regular tetrahedron under the Hopf fibration.  Note that although
 the binary tetrahedral group and the complex spherical reflection
 group $\refl(G_4)$ both have size $24$ and both act freely on the set
 $\Phi$, their actions are distinct since every element of the former
 has a fixed-point free action on all of $\sph^3$ while the
 reflections in the latter pointwise fix complex lines.  Both groups
 can be viewed as index $3$ subgroups of the group of size $72$ that
 stabilizes $\Phi$ setwise and preserves the complex structure, or as
 subgroups of the full isometry group of the $24$-cell of size $1152$,
 also known as the Coxeter group of type $F_4$.
\end{rem}

We use the lens diagram in Figure~\ref{fig:lenses} to understand the
points in the $24$-cell that are fixed by some reflection in
$\refl(G_4)$.

\begin{rem}[Fixed points]\label{rem:g4-fixed-points}
  Consider the reflection $r_1$ in the group $\refl(G_4)$.  It rotates
  the complex line $1\C$ through an angle of $\frac{2\pi}{3}$ and
  pointwise fixes the orthogonal complex line, which in this case is
  the line $(i-j)\C$.  In Figure~\ref{fig:lenses}, each of the six
  lenses is stablized and rotated.  In the top lens, for example, $1$
  goes to $\zeta^2$, which goes to $\zeta^4$, which goes to $1$ and
  $i$ goes to $j\zeta^2$, which goes to $k\zeta^4$, which goes to $i$.
  The fixed portion of each lens is the line segment connecting the
  center of the back hemisphere to the center of the front hemisphere
  through the center of the octahedron.  The six fixed arcs in the six
  lenses glue together to form a single fixed circle or a single fixed
  hexagon, depending on whether this figure is viewed as representing
  the $3$-sphere through the points $\Phi$ or as a slight distortion
  of portions of the piecewise euclidean boundary of the $24$-cell
  with vertices $\Phi$, respectively. The other reflections $r_q$ with
  $q \in \{i,j,k\}$, being conjugates of $r_1$, are geometrically
  similar but their action is slightly harder to see.  Basically,
  $r_q$ rotates the complex line $q\C$ and it cyclically permutes the
  other three complex lines.  Every octahedron contains parts of three
  complex lines in its $1$-skeleton and the six octahedra that contain
  parts of the three other lines form a solid ring or \emph{necklace},
  overlapping on triangles, which contains the circle/hexagon
  orthogonal to the line $q\C$ in its interior as in
  Figure~\ref{fig:necklace}.  Concretely, the fixed hyperplanes for
  the reflections $r_1$, $r_i$, $r_j$ and $r_k$ are $(i-j)\C$,
  $(1+k)\C$, $(1-k)\C$ and $(i+j)\C$, respectively.
\end{rem}

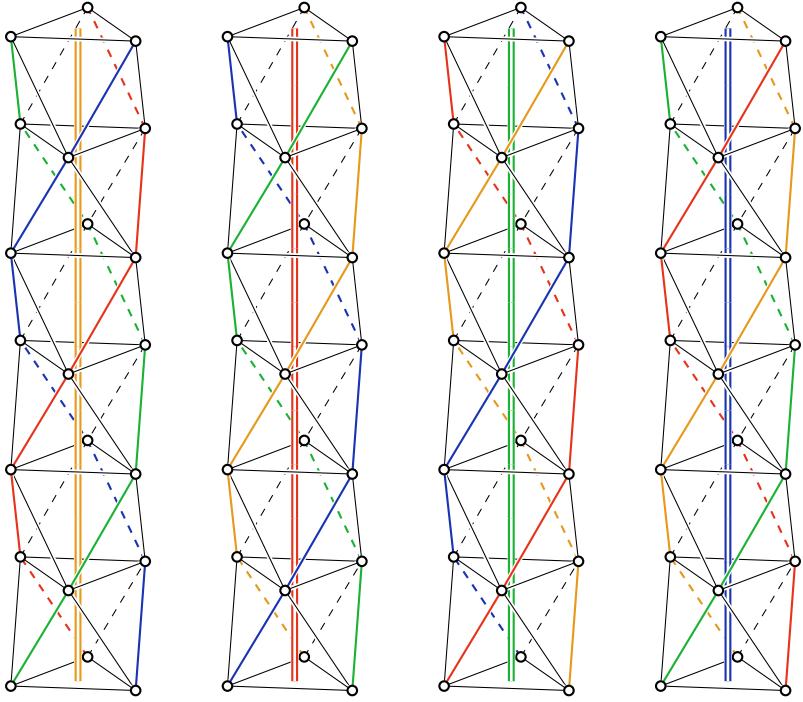
\begin{figure}
  \centering
  \begin{tikzpicture}[scale=.9]
    \newcommand{\xangle}{7}
    \newcommand{\yangle}{145}
    \newcommand{\zangle}{90}
    \newcommand{\xlength}{1}
    \newcommand{\ylength}{0.5}
    \newcommand{\zlength}{1}
    \newcommand{\hei}{1.6}
    \newcommand{\spread}{1.6}
    \pgfmathsetmacro{\xx}{\xlength*cos(\xangle)}
    \pgfmathsetmacro{\xy}{\xlength*sin(\xangle)}
    \pgfmathsetmacro{\yx}{\ylength*cos(\yangle)}
    \pgfmathsetmacro{\yy}{\ylength*sin(\yangle)}
    \pgfmathsetmacro{\zx}{\zlength*cos(\zangle)}
    \pgfmathsetmacro{\zy}{\zlength*sin(\zangle)}
    \pgfmathsetseed{1}
    
    \foreach \cola/\colb/\colc/\cold/\xshif in {
      br1/bg1/bb1/bo1/-3*\spread,
      bo1/bb1/bg1/br1/-1*\spread,
      bb1/br1/bo1/bg1/\spread,
      bo1/bg1/br1/bb1/3*\spread}{
      \begin{scope}[shift={(\xshif,0)}, scale=1,   x={(\xx cm,\xy cm)}, y={(\yx cm,\yy cm)}, z={(\zx cm,\zy cm)},]
        \filldraw[\cold, double, thick, double distance=1pt] (0,0,0)--(0,0,6*\hei);
        \draw (60:1)--(180:1)--(300:1)--cycle;
        \foreach \x/\y/\name in {60/1/2,180/1/0,300/1/1}{\coordinate (v\name0) at (\x:\y);}
        \begin{scope}[shift={(0,0,\hei)}]
          \draw (0:1)--(120:1)--(240:1)--cycle;
          \foreach \x/\y/\name in {0/1/1,120/1/2,240/1/0}{\coordinate  (v\name1) at (\x:\y);}
        \end{scope}
        \begin{scope}[shift={(0,0,2*\hei)}]
          \draw (60:1)--(180:1)--(300:1)--cycle;
          \foreach \x/\y/\name in {60/1/1,180/1/2,300/1/0}{\coordinate  (v\name2) at (\x:\y);}
        \end{scope}
        \begin{scope}[shift={(0,0,3*\hei)}]
          \draw  (0:1)--(120:1)--(240:1)--cycle;
          \foreach \x/\y/\name in {0/1/0,120/1/1,240/1/2}{\coordinate  (v\name3) at (\x:\y);}
        \end{scope}
        \begin{scope}[shift={(0,0,4*\hei)}]
          \draw (60:1)--(180:1)--(300:1)--cycle;
          \foreach \x/\y/\name in {60/1/0,180/1/1,300/1/2}{\coordinate  (v\name4) at (\x:\y);}
        \end{scope}
        \begin{scope}[shift={(0,0,5*\hei)}]
          \draw (0:1)--(120:1)--(240:1)--cycle;
          \foreach \x/\y/\name in {0/1/2,120/1/0,240/1/1}{\coordinate  (v\name5) at (\x:\y);}
        \end{scope}
        \begin{scope}[shift={(0,0,6*\hei)}]
          \draw[white, line width=.1em] (60:1)--(180:1)--(300:1)--cycle;
          \draw (60:1)--(180:1)--(300:1)--cycle;
          \foreach \x/\y/\name in {60/1/2,180/1/0,300/1/1}{\coordinate  (v\name6) at (\x:\y);}
        \end{scope}
        \draw[\colb , thick, dashed] (v03)--(v04)--(v05);
        \draw[\colc , thick, dashed] (v11)--(v12)--(v13);
        \draw[\cola , thick, dashed] (v20)--(v21) (v25)--(v26);
        \draw[dashed] (v21)--(v12)--(v03);
        \draw[dashed] (v13)--(v04)--(v25);
        \draw[dashed] (v20)--(v11) (v05)--(v26);
        \draw[white, line width=.15em] (v00)--(v21) (v03)--(v24)--(v15)--(v06);
        \draw (v00)--(v21) (v03)--(v24)--(v15)--(v06);
        \draw[white, line width=.15em] (v10)--(v01)--(v22)--(v13);
        \draw (v10)--(v01)--(v22)--(v13) (v25)--(v16);
        \draw[white, line width=.15em] (v11)--(v02)--(v23)--(v14)--(v05);
        \draw (v11)--(v02)--(v23)--(v14)--(v05);
        \draw[white, line width=.15em] (v21)--(v22)--(v23)--(v24)--(v25);
        \draw[\cola, thick] (v21)--(v22)--(v23)--(v24)--(v25);
        \draw[white, line width=.15em] (v10)--(v11) (v13)--(v14)--(v15)--(v16);
        \draw[\colc, thick] (v10)--(v11) (v13)--(v14)--(v15)--(v16);
        \draw[white, line width=.15em] (v00)--(v01)--(v02)--(v03);
        \draw[\colb, thick] (v00)--(v01)--(v02)--(v03) (v05)--(v06);
        \foreach \h in {0,2,4,6}{
          \draw[white, line width=.15em] (v0\h)--(v1\h)--(v2\h)--(v0\h);
          \draw(v0\h)--(v1\h)--(v2\h)--(v0\h);}
        \newcommand{\wid}{1.8}
        \draw[white, line width =.15em] (v01)--(v11) (v23)--(v03) (v15)--(v25);
        \draw (v01)--(v11) (v23)--(v03) (v15)--(v25);  
        \foreach \a in {0,1,2}{
          \foreach \b in {0,1,2,3,4,5,6}{
            \filldraw[white] (v\a\b) circle (2pt);
            \draw[black,thick] (v\a\b) circle(2pt);}}
        \foreach \x in {0,2*\hei, 4*\hei}{
          \filldraw[white, line width=3pt] (0,0,\x-.05)--(0,0,\x+.75);
          \filldraw[\cold, double, thick, double distance=1pt] (0,0,\x-.05)--(0,0,\x+.75);}
        \foreach \x in {\hei, 3*\hei,5*\hei}{
          \filldraw[white, line width=3pt] (0,0,\x+.11)--(0,0,\x+.75);
          \filldraw[\cold, double, thick, double distance=1pt] (0,0,\x+.1)--(0,0,\x+.75);}
      \end{scope}
    }
  \end{tikzpicture}
  \caption[The octahedral necklaces]{The $4$ octahedral necklaces
    centered around the fixed orthogonal circles/hexagons are created
    by identifying the top and bottom triangle in each pillar.  The
    triangles in the boundaries of the necklaces can be pairwise
    identified to form the boundary of the $24$-cell homemorphic to a
    $3$-sphere.\label{fig:necklace}}
\end{figure}

The portion of the $24$-cell that avoids the fixed hyperplanes of the
reflections in $\refl(G_4)$ is of particular interest.

\begin{defn}[The complement complex $K_0$]\label{def:K0}
  Let $P$ be the $24$-cell whose vertices are the quaternions in
  $\Phi$ and let $K_0$ be the cell complex formed by the union of the
  faces of $P$ that do not intersect the fixed hyperplanes of the
  reflections in $\refl(G_4)$.  The interior, all $24$ octahedral
  facets, and some of the equilateral triangles are removed while the
  entire $1$-skeleton and some of the triangles remain.  From the
  description of the fixed hyperplanes given in
  Remark~\ref{rem:g4-fixed-points} we see that a triangular face of
  $P$ is excluded precisely when all three of its vertices belong to
  distinct complex lines and it is included when two of the vertices
  belong to the same complex line.  In Figure~\ref{fig:lenses} the
  included triangles can be characterized as those which contain an
  arrow (representing right multiplication by $\zeta$) as one of its
  edges.
\end{defn}

The complex $K_0$ has a number of nice properties including being
non-positively curved.

\begin{rem}[Non-positive curvature]\label{rem:npc}
  We have chosen not to include a detailed review of the notions of
  $\cat(0)$ and non-positive curvature because we only need an easily
  described special case of the theory.  In any piecewise euclidean
  $2$-complex, the link of a vertex is the metric graph of points
  distance $\epsilon$ from the vertex for some small $\epsilon$ that
  is then rescaled so that the length of each arc is equal to the
  radian measure of the angle at the corner of the polygon to which it
  corresponds.  Such a metric graph is said to be $\cat(1)$ when it
  does not contain any simple loop of length strictly less than $2\pi$
  and a piecewise euclidean $2$-complex is called \emph{non-positively
    curved} when every vertex link is $\cat(1)$.  The universal cover
  of a non-positively curved $2$-complex is contractible and it is
  satisfies the definition of being a complete $\cat(0)$ space.
  Finally, a group that acts geometrically on a complete $\cat(0)$
  space is called a \emph{$\cat(0)$ group}.
\end{rem}

\begin{thm}[The complement complex $K_0$]\label{thm:K0}
  The hyperplane complement of $\refl(G_4)$ deformation retracts onto
  a non-positively curved piecewise euclidean $2$-complex $K_0$
  contained in the boundary of the $24$-cell in which every $2$-cell
  is an equilateral triangle and every vertex link is a subdivided
  theta graph.
\end{thm}

\begin{proof}
  The deformation retraction from the hyperplane complement to $K_0$
  comes from our description of how the fixed hyperplanes of the
  reflections in $\refl(G_4)$ intersect the $24$-cell.  More
  explicitly, since the origin belongs to all $4$ fixed hyperplanes,
  we can radially deformation retract the hyperplane complement onto
  the boundary of the $24$-cell, away from the origin (and from
  $\infty$) and the missing hyperplanes correspond to four missing
  hexagons running through the centers of the four solid rings formed
  out of six octahedra each.  See Figure~\ref{fig:necklace}. The
  second step radially deformation retracts from these missing
  hexagons onto the $2$-complex $K_0$.  The punctured triangles
  retract onto their boundary and the pierced octahedra retract on the
  annulus formed by the six triangles which contain an arrow as an
  edge.  Finally, each vertex of $K_0$ is part of $9$ triangles and
  its link is a theta-graph consisting of three arcs of length $\pi$
  sharing both endpoints, subdivided into subarcs of length
  $\frac{\pi}{3}$.  Since the vertex links contain no simple loops of
  length less than $2\pi$, the complex itself is non-positively
  curved.
\end{proof}

As a corollary of Theorem~\ref{thm:K0} we get a detailed description
of the corresponding braid group $\braid(G_4)$.

\begin{cor}[The group $\braid(G_4)$]
  The group $\braid(G_4)$ is a $\cat(0)$ group isomorphic to the
  three-strand braid group and it is defined by the presentation
  $\langle a,b,c,d \mid abd, bcd, cad \rangle.$
\end{cor}

\begin{proof}
  By Steinberg's Theorem the hyperplane complement is the same as the
  space of regular points in this case and by Theorem~\ref{thm:K0} the
  quotient of $K_0$ by the action of $\refl(G_4)$ is homotopy
  equivalent to the space of regular orbits for $\refl(G_4)$.  In
  particular, the fundamental group of the quotient is isomorphic to
  $\braid(G_4)$.  The quotient of the $2$-complex $K_0$ by the free
  action of $\refl(G_4)$ yields a one vertex complex with four edges
  and three equilateral triangles.  The presentation is read off from
  this quotient with the three relations corresponding to the three
  triangles and, once $d$ is solved for and eliminated, the relations
  reduce to $ab=bc=ca$ which is the dual presentation for the
  three-strand braid group.  Finally, since $K_0$ is non-positively
  curved, so is its quotient and its universal cover is $\cat(0)$.
  The free and cocompact action of $\braid(G_4)$ on $\wt K_0$ shows
  that it is a $\cat(0)$ group.
\end{proof}

The fact that the braid group of $\refl(G_4)$ is isomorphic to the
$3$-strand braid group is well-known \cite{Ban76,BMR95,BMR98}.  The
novelty of our presentation is that we use an explicit piecewise
euclidean $2$-complex in the $2$-skeleton in the boundary of the
$24$-cell to establishes this connection.

\section{The group $\refl(\wt G_4)$\label{sec:wtG4}}

This section defines the group $\refl(\wt G_4)$ and establishes key
facts about its translations, its reflections and their fixed
hyperplanes and intersections, as well as the structure of its Voronoi
cells.

\begin{defn}[The group $\refl(\wt G_4)$]
  Let $\refl(\wt G_4)$ denote the group generated by the reflections
  $r_1$, $r_i$ and $r_1' = t_{1+k} \circ r_1 \circ t_{1+k}^{-1} = t_2
  \circ r_1$.  The first two generate $\refl(G_4)$ as before and the
  third, $r_1'(x) = \zeta x \zeta + 2$, is a complex euclidean
  reflection whose action on $\quat_\omega$ is a translated version of
  $r_1$.  The first equation shows that $r_1'$ is a complex euclidean
  reflection fixing $1+k$ and the equality of the two is an easy
  computation.
\end{defn}

One can also write $r'_1= t_{1+i} \circ r_1 \circ t_{1+i}^{-1}$.  Our
choice of $t_{1+k}$ as the conjugating translation is motivated by the
following computation.

\begin{exmp}[An isolated fixed point]\label{ex:1+k}
  The sets $\fix(r_1) = (i-j)\C$, $\fix(r_i) = (1+k)\C$ and
  $\fix(r'_1) = (1+k) + (i-j)\C$ can be described as
  \[\fix(r_1) = \{ a+bi+cj+dk \mid a=0, b+c+d=0\},\]
  \[\fix(r_i) = \{ a+bi+cj+dk \mid b=0, a+c-d=0\},\]  
  and 
  \[\fix(r'_1) = \{ a+bi+cj+dk \mid a=1, b+c+d=1\}.\] 
  Solving these equations, one finds that $1+k \in \Phi_{D_4}$ is the
  unique point in the intersection $\fix(r'_1) \cap \fix(r_i)$.  Thus
  $r_1'$ and $r_i$ generate a copy of $\refl(G_4)$ that uses $1+k$ as
  its origin.
\end{exmp}

The complex spherical reflection group $\refl(G_4)$ acts on the root
system $\Phi$ and the complex euclidean reflection group $\refl(\wt
G_4)$ acts on the Hurwitzian integers they generate.

\begin{defn}[Hurwitzian integers]
  The $\Z$-span of $\Phi$ inside $\quat$ is the set $\Lambda$ of
  \emph{Hurwitizian integers}.  It consists of all quaternions of the
  form $\frac{a+bi+cj+dk}{2}$ where $a$, $b$, $c$ and $d$ are all even
  integers or all odd integers.  Our notation is derived from the
  theory of Coxeter groups.  The $\Z$-span of a root system $\Phi$ is
  is its \emph{root lattice} $\Lambda$ and, as with $\Phi$, we write
  $q \cdot \Lambda$ for the $\Z$-span of $q\cdot \Phi$ and
  $\Lambda_{D_4} = \{ (a,b,c,d) \in \Z^4 \mid a+b+c+d \in 2\Z \}$ for
  the $\Z$-span of $\Phi_{D_4}$.  We note that $2\cdot \Lambda \subset
  \Lambda_{D_4} \subset \Lambda$ and that each is an index $4$ subset
  of the next.
\end{defn}

The Hurwitizian integers $\Lambda$ has many nice properties including
that they form a subring of the quaternions with $\Phi$ as its group
of units, every element has an integral norm, and it satisfies a
noncommutative version of the euclidean algorithm.  See \cite[Chapter
  $5$]{CoSm03} for details.  The remainder of the section records
basic facts about the action of $\refl(\wt G_4)$ on $\quat_\omega$.

\begin{fact}[Translations]\label{fact:translations}
  The translations in the group $\refl(\wt G_4)$ are those of the form
  $t_q$ with $q \in 2 \cdot \Lambda$.
\end{fact}

\begin{proof}
  The element $t_2 = r_1' \circ r_1^{-1}$ is a translation in
  $\refl(\wt G_4)$ and conjugating $t_2$ by elements of $\refl(G_4)$
  shows that all the translations $t_q$ for all $q \in 2 \cdot \Phi$
  are also in $\refl(\wt G_4)$.  Thus $\refl(\wt G_4)$ contains the
  abelian subgroup $T$ that they generate and this consists of all
  translations of the form $2 \cdot \Lambda$.  The subgroup $T$ is
  normal since it is stabilized by the generating set and, because the
  quotient by $T$ is $\refl(G_4)$, the elements in $T$ are the only
  translations in $\refl(\wt G_4)$.
\end{proof}

\begin{fact}[Isolated fixed points]\label{fact:fixed-points}
  There is a copy of $\refl(G_4)$ inside $\refl(\wt G_4)$ fixing a
  point $v$ for each $v$ in the lattice $\Lambda_{D_4}$.  In
  particular, every point in $\Lambda_{D_4}$ is an intersection of
  fixed hyperplanes of complex reflections in $\refl(\wt G_4)$.
\end{fact}

\begin{proof}
  By Example~\ref{ex:1+k} this holds for $v = 1+k$ and if we conjugate
  the copy of $\refl(G_4)$ fixing $1+k$ by an element of the copy fixing
  the origin we find copies fixing $v$ for all $v \in \Phi_{D_4}$.
  Next, conjugating the copy at the origin by elements in the copies
  fixing the points in $\Phi_{D_4}$ shows that there is a copy fixing
  every point that is a sum of two elements in $\Phi_{D_4}$.
  Continuing in this way shows that there is a copy fixing any point
  that is a finite sum of elements in $\Phi_{D_4}$, a set equal to
  $\Lambda_{D_4}$.
\end{proof}

\begin{fact}[Reflections]\label{fact:reflections}
  For every element $v \in \Lambda_{D_4}$ and for every $q \in
  \{1,i,j,k\}$, the primitive complex reflection $t_v \circ r_q \circ
  t_v^{-1}$ is in $\refl(\wt G_4)$.  In fact, these are the only
  primitive complex reflections in $\refl(\wt G_4)$.
\end{fact}

\begin{proof}
  The first assertion is an immediate consequence of
  Fact~\ref{fact:fixed-points}.  When $r' = t_v \circ r_q \circ
  t_v^{-1}$ for some $v$ and for $q\in \{1,i,j,k\}$, we say that $r'$
  is \emph{parallel} to $r_q$.  For the second assertion we note that
  every primitive complex reflection $r'$ in $\refl(\wt G_4)$ must be
  parallel to one of the primitive reflections $r_q$ with $q \in
  \{1,i,j,k\}$ in $\refl(G_4)$.  It is then straight-forward to show
  that if there were an $r'$ in $\refl(\wt G_4)$ parallel to $r_q$
  other than the ones listed, then $r' \circ r_q^{-1}$ would be a
  translation in $\refl(\wt G_4)$ that violates
  Fact~\ref{fact:translations}.
\end{proof}

\begin{fact}[Fixed hyperplanes]
  The fixed hyperplanes of the complex reflections in $\refl(\wt G_4)$
  of the form $t_v \circ r_q \circ t_v^{-1}$ with $v \in
  \Lambda_{D_4}$ and $q \in \{1,i,j,k\}$ can be described as follows
  \[\fix(t_v \circ r_1 \circ t_v^{-1}) = \{ a+bi+cj+dk \mid a=\ell,\
  b+c+d=m\}\] 
  \[\fix(t_v \circ r_i \circ t_v^{-1}) = \{ a+bi+cj+dk \mid b=\ell,\
  a+c-d=m\}\] 
  \[\fix(t_v \circ r_j \circ t_v^{-1}) = \{ a+bi+cj+dk \mid c=\ell,\
  a+d-b=m\}\] 
  \[\fix(t_v \circ r_k \circ t_v^{-1}) = \{ a+bi+cj+dk \mid d=\ell,\
  b+c-a=m\}\] 
  where $\ell$ and $m$ are the unique integers so that $v$ satisfies
  the equations.
\end{fact}

\begin{proof}
  Direct computation.
\end{proof}

Once the reflections and their fixed hyperplanes have been computed,
it is easy to show that the isolated fixed points listed in
Fact~\ref{fact:fixed-points} are the only points that arise as
intersections of fixed hyperplanes.  This set then determines the
structure of the Voronoi cells.

\begin{fact}[Voronoi cells]
  In the Voronoi cell structure around the set of isolated hyperplane
  intersections for the group $\refl(\wt G_4)$ the Voronoi cell around
  the origin is the standard $24$-cell with vertices $\Phi$, the other
  Voronoi cells are translates of the $24$-cell by vectors in
  $\Lambda_{D_4}$ and the link of each vertex in the Voronoi cell
  structure is a $4$-dimensional cube.
\end{fact}

\begin{proof}
  The Voronoi cells for the $D_4$ root lattice is a standard
  computation.  See \cite[Section~7.2]{CoSl99} for details.
\end{proof}

\begin{fact}[Vertices]\label{fact:vertices}
  The vertices of the Voronoi cell structure are located at the points
  in $\Lambda \setminus \Lambda_{D_4}$ and the group $\refl(\wt G_4)$
  acts transitively on this set.
\end{fact}

\begin{proof}
  Since the translates of the $24$-cell are centered at the elements
  of $\Lambda_{D_4}$, every vertex of the Voronoi cell structure can
  be described as $u + v$ with $u \in \Phi$ and $v \in \Lambda_{D_4}$.
  After noting that $\Lambda$ contains $\Lambda_{D_4}$ as a
  sublattice, it is easy to check that every element of $\Lambda$ that
  is not in $\Lambda_{D_4}$ differs from an element of $\Lambda_{D_4}$
  by an element in $\Phi$.  To see transitivity, note that the
  $1$-skeleton of the Voronoi cell structure is connected, each edge
  is in the boundary of one of the $24$-cells, and the local copy of
  $\refl(G_4)$ fixing each $24$-cell acts transitively on its vertices.
\end{proof}

The following key fact is another easy computation.

\begin{fact}[Intersections]\label{fact:intersections}
  If the fixed hyperplane $H$ of a complex reflection in the group
  $\refl(\wt G_4)$ non-trivially intersects is one of the closed
  Voronoi cells, then $H$ contains the point at the center of that
  Voronoi cell.
\end{fact}

\section{Proofs of Main Theorems\label{sec:theorems}}

In this section we prove our three main results.  We begin by defining
the complement complex~$K$.

\begin{defn}[Complement complex $K$]
  The \emph{complement complex $K$} is the portion of the Voronoi cell
  structure for the group $\refl(\wt G_4)$ that is disjoint from the
  union of the fixed hyperplanes of its complex reflections.  By
  Fact~\ref{fact:intersections}, around each fixed hyperplane
  intersection point the portion of $K$ in the boundary of this
  particular $24$-cell is a copy of the $2$-complex $K_0$ defined in
  Definition~\ref{def:K0}.  Thus $K$ can be viewed as a union of local
  copies of~$K_0$.
\end{defn}

The vertex links in $K$ are isomorphic to a well-known graph.

\begin{defn}[M\"obius-Kantor graphs]
  The link of a vertex in the complement complex $K$ is the portion of
  the $1$-skeleton of the $4$-cube shown in
  Figure~\ref{fig:moebius-kantor}. This is a $16$ vertex $3$-regular
  graph known as the \emph{M\"obius-Kantor graph}.  The $8$ removed
  edges correspond to the equilateral triangles whose center lies in
  one of the fixed hyperplanes.  The portion of this graph that lives
  in one of the eight $3$-cubes in the $4$-cube is the subdivided
  theta graph that is the link of this vertex inside the corresponding
  copy of $K_0$ inside a particular $24$-cell.
\end{defn}

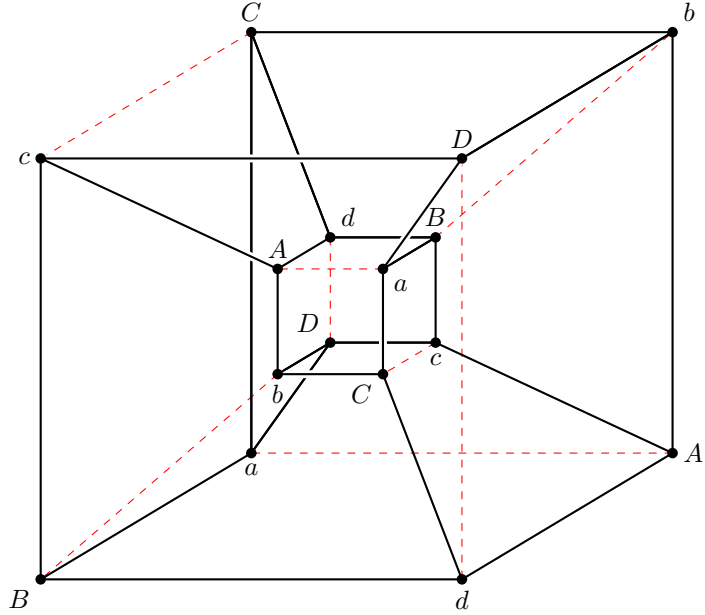
\begin{figure}
  \begin{tikzpicture}[scale=.7,z={(-.5cm,-.3cm)}]
    \def\r{4}
    \coordinate (1) at (1,1,1);
    \coordinate (2) at (1,1,-1);
    \coordinate (3) at (1,-1,-1);
    \coordinate (4) at (1,-1,1);
    \coordinate (5) at (-1,1,1);
    \coordinate (6) at (-1,1,-1);
    \coordinate (7) at (-1,-1,-1);
    \coordinate (8) at (-1,-1,1);
    \coordinate (9) at (\r,\r,\r);
    \coordinate (10) at (\r,\r,-\r);
    \coordinate (11) at (\r,-\r,-\r);
    \coordinate (12) at (\r,-\r,\r);
    \coordinate (13) at (-\r,\r,\r);
    \coordinate (14) at (-\r,\r,-\r);
    \coordinate (15) at (-\r,-\r,-\r);
    \coordinate (16) at (-\r,-\r,\r);
    \draw[red,thin, dashed] (1)--(5) (2)--(10) (3)--(4) (6)--(7)
    (8)--(16) (9)--(12) (11)--(15) (13)--(14);
    \draw[thick] (2)--(6)--(14)--(15)--(7)--(3);
    \draw[thick] (4)--(1)--(2)--(3)
    (7)--(8)--(5)--(6)
    (9)--(10)--(11)--(12)
    (14)--(15)--(16)--(13)
    (5)--(13)--(9)--(1)
    (2)--(6)--(14)--(10)
    (11)--(3)--(7)--(15)
    (16)--(12)--(4)--(8);
    \draw[line width=3pt, color=white] (8)--(4)--(1)--(9)--(13)--(5);
    \draw[thick,black] (8)--(4)--(1)--(9)--(13)--(5) (9)--(10)
    (1)--(2) (7)--(8);
    \foreach \x in {1,2,...,16} {\fill (\x) circle (1mm);}
    \draw (1) node[below right] {\figsize $a$};
    \draw (2) node[above] {\figsize $B$};
    \draw (3) node[below] {\figsize $c$};
    \draw (4) node[below left] {\figsize $C$};
    \draw (5) node[above] {\figsize $A$};
    \draw (6) node[above right] {\figsize $d$};
    \draw (7) node[above left] {\figsize $D$};
    \draw (8) node[below] {\figsize $b$};
    \draw (9) node[above] {\figsize $D$};
    \draw (10) node[above right] {\figsize $b$};
    \draw (11) node[right] {\figsize $A$};
    \draw (12) node[below] {\figsize $d$};
    \draw (13) node[left] {\figsize $c$};
    \draw (14) node[above] {\figsize $C$};
    \draw (15) node[below] {\figsize $a$};
    \draw (16) node[below left] {\figsize $B$};
  \end{tikzpicture}
  \caption{The M\"obius-Kantor graph as a subgraph of the $1$-skeleton
    of a $4$-cube with $8$ edges removed.\label{fig:moebius-kantor}}
\end{figure}

At this point, the proof of our first main theorem is
straight-forward.

\renewcommand{\themain}{\ref{main:complement}}
\begin{main}[Complement complex]
  The hyperplane complement of $\refl(\wt G_4)$ deformation retracts
  onto a non-positively curved piecewise euclidean $2$-complex $K$ in
  which every $2$-cell is an equilateral triangle and every vertex
  link is a M\"obius-Kantor graph.
\end{main}

\begin{proof}
  The proof is essentially the same as that of Theorem~\ref{thm:K0}
  but with the local deformations combined into a global deformation.
  The first step is to radially deformation retract from the removed
  isolated fixed point at the center of each Voronoi cell to its
  boundary, which can be carried out because of
  Fact~\ref{fact:intersections}.  Next, the secondary deformations
  applied to the punctured equilateral triangles and the skewered
  octahedra are compatible regardless of which Voronoi cell one views
  them as belonging to.  Finally, every edge in every vertex link has
  length $\frac{\pi}{3}$ and since M\"obius-Kantor graphs have no
  simple cycles of combinatorial length less than $6$, there are no
  simple loops of length less than $2\pi$, the vertex links are
  $\cat(1)$ and $K$ is non-positively curved.
\end{proof}

\begin{rem}[Other examples]\label{rem:other}
  We should note that when we have attempted to extend our main
  theorems to other complex euclidean reflections groups acting on
  $\C^2$, it is the analog of Fact~\ref{fact:intersections} where
  those attempts have failed.  It is apparently quite common for a
  fixed hyperplane to intersect the boundary of a Voronoi cell without
  passing through its center.  Unless this intersection happens to be
  contained in a different fixed hyperplane that does pass through the
  center of the Voronoi cell, this missing boundary prevents the
  initial deformation retraction onto a portion of the $3$-skeleton of
  the Voronoi cell structure.  
\end{rem}

We now prove a stronger result that immediately implies
Theorem~\ref{main:fixed-points}.

\begin{thm}[Isolated fixed points]
  The points in $\quat_\omega$ stabilized by a non-trivial element of
  the group $\refl(\wt G_4)$ are those in the union of the fixed
  hyperplanes of its complex reflections together with all of the
  vertices of the complement complex $K$.
\end{thm}

\begin{proof}
  Let $T$ be the set of translations in $\refl(\wt G_4)$, let $T_0$ the
  images of the origin under the translations in $T$ and let $FP_A$ be
  the set of points fixed by some element in $\refl(\wt G_4)$ whose
  linear part is the antipodal map.  As in
  Remark~\ref{rem:fixed-points}, the simplification $-(x-v)+v = -x+2v$
  shows that $2\cdot FP_A = T_0$.  Since $T_0 = 2\cdot \Lambda$, $FP_A
  = \Lambda$ and by Fact~\ref{fact:vertices} there is an element of
  order $2$ fixing each vertex of the complement complex $K$.  Since
  the remaining points in $\Lambda$ are contained in the fixed
  hyperplanes (Fact~\ref{fact:fixed-points}), all points fixed by an
  element whose linear part is the antipodal map have been accounted
  for.  To see that the vertices of $K$ are the only isolated points
  with non-trivial stabilizers, suppose that $x$ is a point with a
  non-trivial stabilizer $s$.  If $x$ does not lie in fixed
  hyperplane, the linear part of $s$ must be something other than a
  complex reflection.  The possibilities for its linear part are the
  antipodal map $L_{-1}$, $\pm L_q$ with $q \in \{i,j,k\}$ or $-r_q$
  or $-r_q^2$ with $q \in \{1,i,j,k\}$ but all of these have a power
  equal to the antipodal map $L_{-1}$: the second power of $\pm L_q$
  is the antipodal map and the third power of $-r_q$ and of $-r_q^2$
  is the antipodal map.  In particular, $x$ must be stabilized by an
  element whose linear part is the antipodal map and thus it is one of
  the ones already identified.
\end{proof}

Since the braid group of a group action is defined as the fundamental
group of the space of regular orbits, and the vertices of $K$ are not
regular points, the complement complex $K$ needs to be modified before
it can be used to investigate the group $\braid(\wt G_4)$.

\begin{defn}[Modified complement complex $K'$]\label{def:modified}
  Let $K_1$ be the union of the complement complex $K$ and the
  set of small closed balls of radius $\epsilon>0$ centered at each of
  the vertices of $K$.  Next, let $K_2$ be the metric space obtained
  by removing from $K_1$ the points corresponding to the vertices of
  $K$.  Finally, let $K'$ be the space obtained by removing from $K_1$
  the open balls of radius $\epsilon$ centered at each of the vertices
  of $K$. We call $K'$ the \emph{modified complement complex}.
\end{defn}

In the same way that $K$ is homotopy equivalent to the hyperplane
complement, $K'$ is homotopy equivalent to the space of regular
points.

\begin{prop}[Homotopy equivalences]
  The spaces $K$, $K_1$ and the hyperplane complement are homotopy
  equivalent as are the spaces $K'$, $K_2$ and the space of regular
  points.
\end{prop}

\begin{proof}
  It should be clear that $K$ and $K_1$ are homotopy equivalent as are
  $K_2$ and $K'$.  Moreover, the deformation retractions used to show
  that the hyperplane complement deformation retracts to $K$ can be
  modified to show that it deformation retracts to $K_1$ instead by
  simply stopping the retraction whenever a point is distance
  $\epsilon$ from a vertex.  This modified deformation retraction can
  then be combined with the radial deformation retraction from $K_2$
  to $K'$ to show that the space of regular points (which removes the
  fixed hyperplanes and the vertices of $K$) is homotopy equivalent to
  the modified complex $K'$.
\end{proof}

The action of $\refl(\wt G_4)$ on $K'$ is now free and the fundamental
group of the quotient is, by definition, the group $\braid(\wt G_4)$.

\begin{defn}[Quotient complex]
  Let $G = \refl(\wt G_4)$. Although the action of $G$ on $K$ is not
  free, we can still investigate the properties of the orbifold
  quotient.  Since the action is proper and cellular with trivial
  stabilizers for every cell of positive dimension, the quotient
  remains a $2$-complex.  In this case, the quotient $K/G$ has one
  vertex, four edges and four triangles and it corresponds to the
  presentation $2$-complex of the presentation $\langle a,b,c,d \mid
  abd, bcd, cad, cba \rangle$.  The group defined by this presentation
  is the binary tetrahedral group and the universal cover of the
  orbifold quotient is the $2$-skeleton of the $24$-cell.  Note that
  selecting any $3$ of the $4$ relations produces an infinite group
  isomorphic to the $3$-strand braid group.  The \emph{modified
    quotient complex $K'/G$} is the quotient of $K'$ by the free
  action of $\refl(\wt G_4)$.  To see its structure consider $K_1/G$
  and $K_2/G$.  The former is a modification of $K/G$ where the
  neighborhood of the unique vertex becomes a cone on an $\R P^2$ with
  the vertex as its cone point, and the latter is this space with the
  cone point removed.  Thus $K'/G$ is a copy of $K/G$ with a
  neighborhood of the vertex removed and a real projective plane
  attached in its place.
\end{defn}

The universal cover of the quotient $K'/G$ is the same as the
universal cover of $K'$ and because $2$-spheres are simply connected,
the universal cover of $K'$ is essentially a modified version of the
universal cover of $K$, where the modifications around each vertex are
locally identical to the ones described in
Definition~\ref{def:modified}.  As a consequence we have the following
result.

\begin{thm}[Universal cover]
  The group $\braid(\wt G_4)$ acts geometrically on $\wt K$, the
  $\cat(0)$ universal cover of the complement complex $K$ and the
  vertex stabilizers have size $2$.
\end{thm}

\begin{proof}
  There is a natural free and isometric action of $\braid(\wt G_4)$ on
  $\wt K'$, the universal cover of $K'$ by deck transformations, which
  leads to a proper isometric action of $\braid(\wt G_4)$ on $\wt K$,
  the $\cat(0)$ universal cover of $K$.  The only non-trivial
  stabilizers are, of course, order $2$ and they only occur at the
  vertices of $\wt K$.  Finally, the action is cocompact because the
  quotient of $\wt K$ by the action of $\braid(\wt G_4)$ is equal to
  the quotient of $K$ by $\refl(\wt G_4)$ which is a compact
  $2$-complex with one vertex, four edges and four triangles.
\end{proof}

And this proves our third main result.

\renewcommand{\themain}{\ref{main:braid-group}}
\begin{main}[Braid group]
  The group $\braid(\wt G_4)$ is a $\cat(0)$ group and it
  contains elements of order~$2$.
\end{main}

\providecommand{\bysame}{\leavevmode\hbox to3em{\hrulefill}\thinspace}
\providecommand{\MR}{\relax\ifhmode\unskip\space\fi MR }
\providecommand{\MRhref}[2]{%
  \href{http://www.ams.org/mathscinet-getitem?mr=#1}{#2}
}
\providecommand{\href}[2]{#2}

\end{document}